\documentclass[12pt,notitlepage]{amsart}
\usepackage{latexsym,amsfonts,amssymb,amsmath,amsthm, graphicx}
\usepackage{color}
\usepackage[inner=1.0in,outer=1.0in,bottom=1.0in, top=1.0in]{geometry}

\newcommand{\mz}{\ensuremath{\mathbb Z}}
\newcommand{\mr}{\ensuremath{\mathbb R}}

\newcommand{\half}{\ensuremath{ \frac{1}{2}}}

\newcommand{\intR}{\int_{-\infty}^{\infty}}

\newcommand{\thalf}{\tfrac12}
\newcommand{\sumstar}{\sideset{}{^*}\sum}

\newcommand{\SL}[2]{SL_{#1}(\mathbb{#2})}
\newcommand{\GL}[2]{GL_{#1}(\mathbb{#2})}
\newcommand{\smax}{\overline{s}_-}
\newcommand{\smin}{\underline{s}_+}

\theoremstyle{plain}		
	\newtheorem{mytheo}{Theorem}[section]
	
	\newtheorem{myprop}[mytheo]{Proposition}
	\newtheorem{mycoro}[mytheo]{Corollary}
     \newtheorem{mylemma}[mytheo]{Lemma}
	
	\newtheorem{myconj}[mytheo]{Conjecture}
	\newtheorem{myremark}[mytheo]{Remark}

\theoremstyle{remark}

\numberwithin{equation}{section}
\numberwithin{figure}{section}
\begin{document}
\title{The $L^2$ restriction norm of a Maass form on $SL_{n+1}(\mathbb{Z})$.}
\author{Xiaoqing Li}
\address{Deptartment of Mathematics \\
State University of New York at Buffalo \\
Buffalo, NY, 14260
}
\email{XL29@buffalo.edu}
\author{Sheng-Chi Liu} 
\author{Matthew P. Young} 
\address{Department of Mathematics \\
	  Texas A\&M University \\
	  College Station \\
	  TX 77843-3368 \\
		U.S.A.}
\email{scliu@math.tamu.edu}
\email{myoung@math.tamu.edu}
\thanks{This material is based upon work supported by the National Science Foundation under agreement Nos. DMS-0901035 (X.L.), DMS-1101261 (M.Y.).  Any opinions, findings and conclusions or recommendations expressed in this material are those of the authors and do not necessarily reflect the views of the National Science Foundation.  The first-named author was also partially supported with an Alfred P. Sloan Foundation Fellowship}

\begin{abstract}
We discuss upper and lower bounds for the $L^2$ restriction norm of a Maass form on $\SL{n+1}{Z}$.
\end{abstract}

\maketitle

\section{Introduction}
It is a basic question in quantum chaos to understand how eigenfunctions of the Laplacian behave as the eigenvalue becomes large.  It is natural to study the $L^p$-norm of the eigenfunctions restricted to a submanifold of the original domain.  For very general results along these lines, see \cite{BGT} and \cite{Marshall}.
In some particularly interesting cases, these restriction norms are related to $L$-functions.  Here we study such a case where we consider a Maass form on $\SL{n+1}{Z} \backslash \mathcal{H}^{n+1}$ (with $n \geq 1$ arbitrary) restricted to $(\SL{n}{Z} \backslash \mathcal{H}^{n}) \times \mathbb{R}^+$, where here $\mathcal{H}^m = SL_{m}(\mr)/SO_m(\mr)$.  This leads to a family of $GL_{n+1} \times GL_{n}$ $L$-functions.  The space $\mathcal{H}^{n+1}$ is a real manifold of dimension $\frac{n(n+3)}{2}$ while $\mathcal{H}^n \times \mathbb{R}^+$ has $\frac{(n-1)(n+2)}{2} + 1$ dimensions, so this is a codimension $n$ restriction.  

For a given manifold, submanifold, and choice of $L^p$-norm to measure, it is not known what to expect the size of the $L^p$-norms to be.  For example, on the space $\SL{2}{\mz} \backslash \mathbb{H}$, it is known that the supremum norm can be $\gg \lambda^{1/12 - \varepsilon}$ \cite{SarnakMorowetz} though the point where this value is attained is high in the cusp (and changes with the eigenvalue).  On the other hand, if one restricts to the point $i$ (or any other fixed Heegner point) then the Waldspurger/Zhang \cite{Waldspurger} \cite{Zhang} formula relates the $L^2$-norm of a Maass form to the central values of a Rankin-Selberg $L$-function, and the Lindel\"{o}f Hypothesis gives a best-possible upper bound of $\lambda^{\varepsilon}$ here.  There are many interesting subtleties in understanding the sizes of these restriction norms \cite{SarnakMorowetz} \cite{SarnakReznikov} \cite{Milicevic} \cite{Templier}.  In a case where the restriction norm is related to a mean value of nonnegative $L$-functions, then the Lindel\"{o}f hypothesis should reveal its size.  However, the present example gives a case where even assuming the Lindel\"{o}f hypothesis, it is not immediately clear what it implies about the restriction norm.  It is only after some intricate combinatorial arguments that one can deduce the size of the restriction norm.  As part of this analysis, we compute the volumes of a certain parametrized family of $n$-dimensional polytopes (see Section \ref{section:polytope} below).  More than this, our application requires the estimation of an integral of a more complicated function over such a polytope. 

We set some notation in order to describe our results.  Let $F(z)$ be an even Hecke-Maass form for $\SL{n+1}{Z}$ with Fourier coefficients $A_F(m_1, \dots, m_n)$ and Langlands parameters $(i\alpha_1, \dots, i\alpha_{n+1})$ with $\alpha_j \in \mr $ for all $j$ (meaning the form is tempered) which satisfy $\alpha_1 + \dots + \alpha_{n+1} = 0$.  Suppose that $F$ is $L^2$-normalized on $\SL{n+1}{Z} \backslash \mathcal{H}^{n+1}$, and set
\begin{equation}
\label{eq:normdef}
 N(F) = \int_0^{\infty} \int_{SL_n(\mathbb{Z}) \backslash \mathcal{H}^n} \Big| F\begin{pmatrix} z_2 y & \\ & 1 \end{pmatrix} \Big|^2 d^*z_2 \frac{dy}{y},
\end{equation}
where $z_2$ is defined by \eqref{eq:z2def} below and $d^* z_2$ is the left-invariant $SL_n(\mathbb{R})$ measure on $\mathcal{H}^n$ (see Proposition 1.5.3 of \cite{Goldfeld}) given by
\begin{equation}
 d^*z_2 = \prod_{1 \leq i < j \leq n} d x_{i,j} \prod_{k=1}^{n-1} y_{k+1}^{-k(n-k)-1} dy_{k+1}.
\end{equation}
If $F$ is odd (so that $n+1$ is even, by Proposition 9.2.5 of \cite{Goldfeld}) then $N(F) = 0$, since the restriction of $F$ is invariant under $SL_{n}(\mz)$, and also odd, and hence zero by the same argument.

Let $\lambda(F)$ be the Laplace eigenvalue of $F$, i.e.,
\begin{equation}
 \lambda(F) = \frac{(n+1)^3-(n+1)}{24} + \half(\alpha_1^2 + \dots + \alpha_{n+1}^2).
\end{equation}

Our main goal here is to determine the size of $N(F)$.  For $n=1$, it is known that $1 \ll N(F) \ll \lambda(F)^{\varepsilon}$; see Theorem 6.1 of \cite{GRS}.  For $n=2$, \cite{LiYoung} showed $N(F) \ll \lambda(F)^{\varepsilon}$ under the assumption that
\begin{equation}
 \label{eq:FirstfourierUB}
 |A_F(1)|^2 \ll \lambda(F)^{\varepsilon}.
\end{equation}
Alternatively, one could assume a lower bound on the residue of $L(s, F \times \overline{F})$ at $s=1$; see Proposition \ref{prop:normformula} below.  This is a difficult problem amounting to showing the non-existence of a Landau-Siegel zero for the Rankin-Selberg $L$-function $L(s, F \times \overline{F})$.  For $GL_2 \times GL_2$, such a bound was famously shown by Hoffstein-Lockhart \cite{HL}.
The condition \eqref{eq:FirstfourierUB} is a consequence of the generalized Riemann hypothesis, but is also a consequence of the Langlands functoriality conjectures; see Theorem 5.10 of \cite{IK}.  
For $n \geq 3$, it seems to be quite ambitious to prove a strong bound on $N(F)$.  In fact, even with standard, powerful assumptions in number theory such as Langlands and Lindel\"{o}f, it is {\em still} challenging to see what is the size of $N(F)$.  As such, we shall undertake our investigations aided by some additional assumptions that we now describe.  

Besides \eqref{eq:FirstfourierUB} which was already needed for $n=2$, we also suppose that all Hecke-Maass forms are {\em tempered}, meaning that $\alpha_j \in \mr$ for all $j$.  If $F$ is an $\SL{n+1}{Z}$ Hecke-Maass form, we require temperedness of all Hecke-Maass forms on $\SL{m}{Z}$ with $2 \leq m \leq n$ (note this is known for $m=2$ unconditionally).  The Langlands functoriality conjectures imply temperedness \cite[Section 8]{Langlands}.
We also require what we shall call the {\em weighted local Weyl law} which we describe in more detail in Section \ref{section:Weyl} below; specifically see \eqref{eq:WLWLlower} and \eqref{eq:WLWLupper}.  The local Weyl law amounts to an estimate for the number of Langlands parameters $i \beta = (i\beta_1, \dots, i \beta_m)$ that lie in a box $\| \beta - \lambda \| \leq 1$ with $\lambda = (\lambda_1, \dots, \lambda_m) \in \mr^m$, $\lambda_1 + \dots + \lambda_m = 0$.  The {\em weighted} local Weyl law gives an estimate for the sum of such Langlands parameters $i\beta$ weighted by the first Fourier coefficient $|A_{\beta}(1)|^2$ of the corresponding Hecke-Maass form.  For $m=2$ this is a standard application of the Kuznetsov formula; Blomer \cite{Blomer} recently obtained the weighted local Weyl law for $m=3$ but for $m \geq 4$ this is open.
The weights occuring in the weighted local Weyl law are quite natural because they appear in the Kuznetsov formula which has seen extensive applications in number theory.

\begin{mytheo}
\label{thm:normUB}
 Assuming the generalized Lindel\"{o}f Hypothesis for Rankin-Selberg $L$-functions, temperedness of $F$ and for all Hecke-Maass forms on $\SL{m}{Z}$ with $2 \leq m \leq n$, the first Fourier coefficient bound \eqref{eq:FirstfourierUB}, and the weighted local Weyl law \eqref{eq:WLWLupper}, we have
\begin{equation}
\label{eq:NFUB}
 N(F) \ll_{n,\varepsilon} \lambda(F)^{\varepsilon}.
\end{equation}
\end{mytheo}
\begin{mytheo}
\label{thm:normLB}
 Assuming temperedness for $F$ and for all Hecke-Maass forms on $\SL{m}{Z}$ with $2 \leq m \leq n$, the weighted local Weyl law \eqref{eq:WLWLlower}, and the spacing condition $|\alpha_j - \alpha_k| \geq \lambda(F)^{\varepsilon}$ for all $j \neq k$, we have
\begin{equation}
 N(F) \gg_{n,\varepsilon} \lambda(F)^{-\varepsilon}.
\end{equation}
\end{mytheo}
Taken together, Theorems \ref{thm:normUB} and \ref{thm:normLB} largely pin down the size of $N(F)$.  The condition that $|\alpha_j - \alpha_k| \geq \lambda(F)^{\varepsilon}$ in the lower bound can probably be removed; see Section \ref{section:walls}.  The spacing condition is relevant because the upper bound on the supremum norm of $F$ becomes smaller when the $\alpha_j$'s are closely spaced; see p.42 of \cite{SarnakMorowetz}.  We chose to present the full details of the proof under the spacing assumption for simplicity of exposition.  

In Theorem \ref{thm:normLB} we do not require \eqref{eq:FirstfourierUB}; rather, we use the lower bound $|A_F(1)|^2 \gg \lambda(F)^{-\varepsilon}$ which is a consequence of the convexity bound for Rankin-Selberg $L$-functions proved in general by Xiannan Li \cite{XiannanLi}.  In fact, our proof of Theorem \ref{thm:normLB} shows the stronger bound $N(F) \gg \lambda(F)^{-\varepsilon} |A_F(1)|^2$, and hence if \eqref{eq:NFUB} holds, we \emph{deduce} \eqref{eq:FirstfourierUB}.  In this way, we see that the upper bound on $N(F)$ is inextricably linked with an upper bound on the first Fourier coefficient.  
Note that Theorem \ref{thm:normLB} is unconditional for $n=2$.

We conjecture that the right order of magnitude of $N(F)$ is given by
\begin{myconj}
\label{conj:asymptotic}
 Based on the conjectures of \cite{CFKRS}, we conjecture
\begin{equation}
 N(F) = C_n(\alpha) \log \lambda(F) + o(\log \lambda(F)),
\end{equation}
where $C_n(\alpha)$ is a function of the Langlands parameters of $F$ which satisfies $C_n(\alpha) \asymp_n 1$.
\end{myconj}
It is not clear for $n > 1$ if $C_n(\alpha) \sim_n C_n$ for some constant $C_n$ independent of $\alpha$; see \eqref{eq:Calpha} for the form of $C_n(\alpha)$ which is an $n$-fold integral involving ratios of gamma functions.

The previous results all use as a starting point a formula for $N(F)$ in terms of Rankin-Selberg $L$-functions.  Specifically, we apply the $GL_n$ spectral decomposition to \eqref{eq:normdef} and calculate the resulting integrals in terms of this family of $L$-functions; this is given in Proposition \ref{prop:normformula}.  The Archimedean factors in this formula, crucially calculated by Stade in terms of gamma functions \cite{StadeAJM},
govern the practical parameterization of the family of $L$-functions.  Of course, the sum over the $GL_n$ spectrum and the integral on the right hand side of \eqref{eq:normformula} is infinite but except for a finite region, the Archimedean factors are exponentially small (following from Stirling's formula) and do not contribute to $N(F)$ in a practical sense.  We carry out this analysis in Section \ref{section:Archimedean}.
This region turns out to be closely related to a problem in representation theory, namely, how an irreducible, finite-dimensional representation of $\GL{n+1}{C}$ decomposes into irreducibles when restricted to $\GL{n}{C}$.  We explain this following the proof of Lemma \ref{lemma:rzero}; somehow the gamma factors in \eqref{eq:qtalphabetaDEF} are analytically detecting this decomposition.  Using this decomposition and the Lindel\"{o}f hypothesis, we are able to deduce Theorem \ref{thm:normUB} in Section \ref{section:upperbound}.  The formula \eqref{eq:normformula} below involves a sum over the $GL_n$ spectrum, expressed in terms of the Langlands parameters $\beta = (\beta_1, \dots, \beta_n)$ which live on the hyperplane $\beta_1 + \dots + \beta_n = 0$, as well as a $t$-integral over the real line.  When combined, the relevant region becomes an $n$-dimensional box with sides parallel to the standard basis vectors of $\mathbb{R}^n$.

Our proof of Theorem \ref{thm:normLB} is more difficult than that of Theorem \ref{thm:normUB}.  The underlying reason is that we may use Lindel\"{o}f to obtain a uniform upper bound on the $L$-functions in the family, but there does not exist a uniform lower bound (of course $L$-functions are expected to have zeros on the critical line!).  Instead, we rely on a very soft argument: on any sufficiently long interval (say at least some small power of the analytic conductor), the second moment of an individiual $L$-function in the $t$-aspect is at least half the length of the interval; for a precise statement, see Proposition \ref{prop:secondmomentLowerBound}.  To use this argument, we need to understand the relevant set of $\beta$'s for which the $t$-integral is long enough to use this lower bound.
Instead of obtaining a box in $\mathbb{R}^n$, we obtain a convex polytope.  This region is described in Section \ref{section:polytope}.  This polytope has a very special structure and we show that it is a zonotope which is in fact naturally given as an affine projection of the $A_n$ lattice.  Using this structure of the polytope, we are able to complete the proof of Theorem \ref{thm:normLB} in Section \ref{section:lowerbound}.  In Section \ref{section:walls} we briefly discuss relaxing the spacing condition $|\alpha_j - \alpha_k| \geq \lambda(F)^{\varepsilon}$ which appears in Theorem \ref{thm:normLB}.  Finally, in Section \ref{section:asymptotic} we discuss Conjecture \ref{conj:asymptotic}.

Throughout the paper we often view $n$ as fixed and we may not always display the dependence of implied constants on $n$.  We also use the common convention of letting $\varepsilon >0$ vary from line-to-line.

\section{Maass forms and period integrals}
We assume some familiarity with Goldfeld's book \cite{Goldfeld}.  Our goal in this section is to produce a formula for $N(F)$ in terms of Rankin-Selberg $L$-functions which we give in Proposition \ref{prop:normformula} below.

Let $u_j(z)$ be a Hecke-Maass form for $\Gamma:=SL_n(\mathbb{Z})$ with Fourier coefficients $B_j(m_2, \dots, m_n)$ and Langlands Parameters $(i\beta_{1,j}, \dots, i\beta_{n,j})$ with $\text{Re}(i\beta_{i,j}) = 0$ for all $i$.  Set $A_F(m_1, \dots, m_n) = A_F(1,\dots, 1) \lambda_F(m_1, \dots, m_n)$ and similarly $B_j(m_2, \dots, m_n) = B_j(1, \dots, 1) \lambda_j(m_2, \dots, m_n)$.  For brevity we sometimes write $A_F(1,\dots, 1) =: A_F(1)$ and $B_j(1,\dots, 1) =: B_j(1)$.  We may assume each Maass form is either even or odd according to if $\lambda(m_1, \dots, - m_q) = \pm \lambda(m_1, \dots, m_q)$; see Proposition 9.2.6 in \cite{Goldfeld}.
The Rankin-Selberg $L$-function on $GL_{n+1} \times GL_n$ is given by
\begin{equation}
\label{eq:RankinSelbergCuspidal}
 L(s, F \times \overline{u_j}) = \sum_{m_1 \geq 1} \dots \sum_{m_{n} \geq 1} 
\frac{\lambda_F(m_1, \dots, m_n) \overline{\lambda_j}(m_2, \dots, m_n)}{\prod_{k=1}^{n} m_k^{(n+1-k)s}}. 
\end{equation}
Let
\begin{equation}
 P_{\text{min}} := \left\{ \begin{pmatrix} * & * & \dots & * \\ & * & \dots & * \\  &  & \ddots & * \\  &  &  & * \end{pmatrix} \in \GL{n}{R} \right\},
\end{equation}
and define the minimal parabolic Eisenstein series
\begin{equation}
 E_{P_{\text{min}}}(z, iw) = \sum_{\gamma \in P_{\text{min}} \cap \Gamma \backslash \Gamma} I_{iw}(\gamma z),
\end{equation}
with
\begin{equation}
\label{eq:Ifunction}
 I_{iw}(z) = \prod_{j=1}^{n-1} \prod_{k=1}^{n-1} y_j^{ b_{j,k} i w_k}, \qquad b_{j,k} = \begin{cases} jk, \quad &\text{if } j + k \leq n \\ (n-j)(n-k), \quad &\text{if } j+k \geq n.                                                                                                                                                     \end{cases}
\end{equation}
Here $\text{Re}(i w_k) = \frac{1}{n}$, the Langlands parameters of $E_{P_{\text{min}}}$ are $(i\beta_{1,w}, \dots, i\beta_{n,w})$ with $\text{Re}(i\beta_{j,w}) = 0$ for all $j$, and satisfying
\begin{equation}
i w_k = \frac{1}{n}(1 + i\beta_{k,w} - i\beta_{k+1,w}), \qquad 1 \leq k \leq n-1.
\end{equation}
Let $B_{P_{\text{min}}}(m_2, \dots, m_n) = B_{P_{\text{min}}}(1,\dots, 1) \lambda_{P_{\text{min}}}(m_2, \dots, m_n)$ be the non-degenerate Fourier coefficients of $E_{P_{\text{min}}}$, and define
\begin{equation}
\label{eq:RankinSelbergEisenstein}
 L(s, F \times \overline{E}_{P_{\text{min}}}) = 
  \sum_{m_1 \geq 1} \dots \sum_{m_{n-1} \geq 1} \sum_{m_n \geq 1} \frac{\lambda_F(m_1, \dots, m_n) \overline{\lambda_{P_{\text{min}}}}(m_2, \dots, m_n)}{\prod_{k=1}^{n} |m_k|^{(n+1-k)s}}. 
\end{equation}
For $r \geq 2$ let $P = P_{n_1, \dots, n_r}$ be a standard parabolic subgroup of $\GL{n}{R}$ associated to the partition $n=n_1 + \dots + n_r$, i.e.,
\begin{equation}
 P =NM = \left\{ \begin{pmatrix} I_{n_1} & * & \dots & * \\ & I_{n_2} & \dots & * \\  &  & \ddots & * \\  &  &  & I_{n_r} \end{pmatrix} \begin{pmatrix} m_{n_1} &  &  &  \\ & m_{n_2} &  &  \\  &  & \ddots &  \\  &  &  & m_{n_r} \end{pmatrix} \in GL_n(\mathbb{R}) \right\},
\end{equation}
where $I_k$ is the $k\times k$ identity matrix and $m_k \in \GL{k}{R}$.  Let $\phi_j = (\phi_{j_1}, \dots \phi_{j_r})$ be a vector of $r$ Hecke-Maass forms where $\phi_{j_k}$ with $1\leq k \leq r$ runs through an orthogonal basis of $C_{\text{cusp}}(\SL{n_k}{Z})$ with first Fourier coefficient equal to $1$, and the Langlands parameters of $\phi_{j_k}$ are 
\begin{equation}
(i\beta_{j_k, \eta_k +1}, \dots, i\beta_{j_k, \eta_k + n_k}), \quad \text{where } \eta_1 =0, \quad \eta_k = n_1 + \dots + n_{k-1},\text{ for } k > 1.
\end{equation}
Note that $\eta_k + n_k = \eta_{k+1}$ for $k \geq 1$.
  Then for $v = (v_1, \dots, v_r) \in \mathbb{C}^r$ with $\sum_{k=1}^{r} n_k v_k = 0$ define the cuspidal Eisenstein series
\begin{equation}
 E_P(z, iv, \phi_j) = \sum_{\gamma \in P \cap \Gamma \backslash \Gamma} \prod_{k=1}^{r} \phi_{k,j}(m_{n_k}(\gamma z)) I_{iv}(\gamma z, P)
\end{equation}
as in \cite{Goldfeld}, Definition 10.5.3.  Also assume
\begin{equation}
 \text{Re}(v_k + \eta_k + \frac{n_k - n}{2}) = 0,
\end{equation}
and let
\begin{equation}
 i v_k^* = v_k + \eta_k + \frac{n_k - n}{2}.
\end{equation}
Notice in \cite{Goldfeld} p.318 Proposition 10.9.3, $s_k + \eta_k$ should be $s_k + \eta_k + \frac{n_k-n}{2}$.  The Langlands parameters of $E_{P_{n_1, \dots, n_r}}$ are the components of $i\beta$ where
\begin{equation}
\label{eq:LanglandsEisenstein}
\beta = (v_1^* + \beta_{j_1,1}, \dots, v_1^*+ \beta_{j_1, n_1} \big| v_2^* + \beta_{j_2, n_1 + 1}, \dots, v_2^* + \beta_{j_2, n_1 + n_2} \big| \dots). 
\end{equation}
Here the notation indicates that $\beta$ has $n$ components, broken into $r$ blocks of size $n_1, n_2, \dots, n_r$; the vertical lines separate these blocks.
Let 
\begin{equation}
B_{{P_{n_1, \dots, n_r}}}(m_2, \dots, m_n) = B_{{P_{n_1, \dots, n_r}}}(1, \dots, 1)\lambda_{{P_{n_1, \dots, n_r}}}(m_2, \dots, m_n)
\end{equation}
 be the non-degenerate Fourier coefficients of $E_{P_{n_1, \dots, n_r}}$.  As before, one can define the Rankin-Selberg $L$-function as in \eqref{eq:RankinSelbergEisenstein}.

Our formula for $N(F)$ requires the following definitions.  Recall that $F$ is even.  Let
\begin{equation}
 \mathcal{L}(s, F \times \overline{u_j}) = 2\overline{B_j}(1) A_F(1) L(s, F \times \overline{u_j}) G_j(s),
\end{equation}
if $u_j$ is even, and $\mathcal{L}(s, F \times \overline{u_j}) = 0$ if $u_j$ is odd, with
\begin{multline}
 G_j(s) = G_j^*(s) \pi^{\frac{n^2}{2} - \sum_{l=1}^{n-1} \sum_{l \leq k \leq n-1} v_{l,k,j} - \sum_{l=1}^{n} \sum_{l \leq k \leq n} v_{l,k}^*}
\\
 \Big(\prod_{1 \leq k < l \le n+1} \Gamma(\frac{1+ i\alpha_k - i\alpha_l}{2}) \Big)^{-1} \Big(\prod_{1 \leq k < l \le n} \Gamma(\frac{1- i\beta_{k,j} + i\beta_{l,j}}{2}) \Big)^{-1},
\end{multline}
\begin{equation}
 G_j^*(s) = 2^{-n} \pi^{-\frac{n(n+1)}{2} s} \prod_{l=1}^{n} \prod_{k=1}^{n+1} \Gamma(\frac{s- i\alpha_k + i\beta_{l,j}}{2}),
\end{equation}
and with
\begin{equation}
 v_{l,k,j} = 
 \frac{i}{2}(\beta_{n-k,j} - \beta_{n-k+l, j}),
\quad
 v_{l,k}^* = 
\frac{i}{2}(\alpha_{n+1-k} - \alpha_{n+1-k+l}).
\end{equation}
Similarly, 
\begin{equation}
 \mathcal{L}(s, F \times \overline{E}_{P_{n_1, \dots, n_r}}) = 2\overline{B}_{P_{n_1, \dots, n_r}}(1) A_F(1) L(s, F \times \overline{E}_{P_{n_1, \dots, n_r}}) G_{j,P}(s),
\end{equation}
if $E_{P_{n_1, \dots, n_r}}$ is even (and vanishes if it is odd),
with
\begin{multline}
 G_{j,P}(s) = G_{j,P}^*(s) \pi^{\frac{n^2}{2} - \sum_{l=1}^{n-1} \sum_{l \leq k \leq n-1} v_{l,k,P_{n_1, \dots, n_r}} - \sum_{l=1}^{n} \sum_{l \leq k \leq n} v_{l,k}^*}
 \prod_{1 \leq k < l \le n+1} \Gamma(\frac{1+ i\alpha_k - i\alpha_l}{2})^{-1} 
\\
\Big(\prod_{1 \leq k_1 \leq k_2 \leq r} \mathop{\prod_{1 \leq l_1 \leq n_{k_1}} \prod_{1 \leq l_2 \leq n_{k_2}} }_{1 \leq \eta_{k_1} + l_1 < \eta_{k_2} + l_2 \leq n} \Gamma(\frac{1- i(v_{k_1}^* + \beta_{j_{k_1}, \eta_{k_1} + l_1}) + i(v_{k_2}^* + \beta_{j_{k_2}, \eta_{k_2} + l_2})}{2}) \Big)^{-1},
\end{multline}
where
\begin{equation}
 G_{j,P}^*(s) = 2^{-n} \pi^{-\frac{n(n+1)}{2} s} \prod_{m=1}^r \prod_{l=1}^{n_m} \prod_{k=1}^{n+1} \Gamma(\frac{s- i\alpha_k + i(v_m^* + \beta_{j_m, \eta_m+l})}{2}),
\end{equation}
and with
\begin{equation}
 v_{l,k,P_{n_1, \dots, n_r}} = \frac{i}{2} (\beta_{n-k, P_{n_1,\dots, n_r}} - \beta_{n-k+l, P_{n_1, \dots, n_r}}),
\end{equation}
where $i\beta_{m, P_{n_1, \dots, n_r}}$ is the $m$-th Langlands parameter of $E_{P_{n_1, \dots, n_r}}$.  Finally, 
\begin{equation}
 \mathcal{L}(s, F \times \overline{E}_{min}(\cdot, w)) = 2\overline{B}_{P_{min}}(1) A_F(1) L(s, F \times \overline{E}_{P_{min}}(\cdot, w)) G_{min}(s)
\end{equation}
if $E_{P_{min}}$ is even (and vanishes if it is odd),
with
\begin{multline}
 G_{min}(s) = G_{min}^*(s) \pi^{\frac{n^2}{2} - \sum_{l=1}^{n-1} \sum_{l \leq k \leq n-1} v_{l,k,P_{min}} - \sum_{l=1}^{n} \sum_{l \leq k \leq n} v_{l,k}^*}
\\
 \Big(\prod_{1 \leq k < l \le n+1} \Gamma(\frac{1+ i\alpha_k - i\alpha_l}{2}) \Big)^{-1} \Big(\prod_{1 \leq k < l \le n} \Gamma(\frac{1- i\beta_{k,w} + i\beta_{l,w}}{2}) \Big)^{-1},
\end{multline}
\begin{equation}
 G_{min}^*(s) = 2^{-n} \pi^{-\frac{n(n+1)}{2} s} \prod_{l=1}^{n} \prod_{k=1}^{n+1} \Gamma(\frac{s- i\alpha_k + i\beta_{l,w}}{2}),
\end{equation}
\begin{equation}
 v_{l,k,P_{min}} 
 = \frac{i}{2}(\beta_{n-k,w} - \beta_{n-k+l, w}),
\end{equation}
where recall $(i\beta_{1,w}, \dots, i \beta_{n,w})$ are the Langlands parameters of $E_{P_{min}}$.

\begin{myprop}
\label{prop:normformula}
 We have
\begin{multline}
\label{eq:normformula}
N(F) = \frac{n}{2\pi} \sum_{\substack{\SL{n}{Z} \\ \text{ cuspidal spectrum}}} \intR |\mathcal{L}(1/2 +it, F \times \overline{u_j})|^2 dt
\\
+ 
\sum_{\substack{\SL{n_1}{Z} \\ \text{ cuspidal spectrum}}} 
\dots
\sum_{\substack{\SL{n_r}{Z} \\ \text{ cuspidal spectrum}}} c_{n_1, \dots, n_r} \int_{\mr^r} |\mathcal{L}(1/2 + it, F \times \overline{E}_{P_{n_1, \dots, n_r}}(\cdot, iv, \phi_j)|^2 dt dv_1^* \dots dv_{r-1}^*
\\
+ c\intR \dots \intR |\mathcal{L}(1/2 + it, F \times \overline{E}_{P_{\text{min}}}(\cdot, iw)|^2 dt d \beta_{1,w} \dots d \beta_{n-1,w},
\end{multline}
with the middle sum running over all partitions $n_1 + \dots + n_r = n$ where each $n_i \geq 1$, and $c_{n_1 \dots, n_r}$ and $c$ are certain positive constants; we take the convention that if $n_i = 1$ then we take the constant eigenfunction.  
\end{myprop}
To prove Proposition \ref{prop:normformula}, we need the following lemmas.
\begin{mylemma}
 For fixed $y > 0$,
\begin{equation}
 f_y(z_2) := F\begin{pmatrix} z_2 y & \\ & 1 \end{pmatrix} \in L^2(\SL{n}{Z} \backslash \mathcal{H}^n),
\end{equation}
where
\begin{equation}
\label{eq:z2def}
 z_2 := 
\begin{pmatrix} 1 & x_{1,2} & \dots & x_{1,n} \\ & 1 & \dots & x_{2,n} \\  &  & \ddots &  \\  &  &  & 1 \end{pmatrix} Y, \quad Y=\begin{pmatrix} y_2 \dots y_n &  &  & &  \\ & y_2 \dots y_{n-1} &  & & \\  &  & \ddots & & \\  &  & & y_2 &   \\  &  & &  & 1 \end{pmatrix} 
\prod_{k=2}^{n} y_k^{-\frac{n+1-k}{n}}
\end{equation}
\end{mylemma}
\begin{proof}
 Since $F$ is a Maass form for $\SL{n+1}{Z}$, it has rapid decay when $y_k \rightarrow \infty$, $2 \leq k \leq n$.  
\end{proof}

\begin{mylemma}
\label{lemma:ujFinnerproduct}
 Let
\begin{equation}
 \mathcal{L}(s, F \times \overline{u_j}) := \int_0^{\infty} \int_{\SL{n}{Z} \backslash \mathcal{H}^n} \overline{u_j}(z_2)  F\begin{pmatrix} z_2 y & \\ & 1 \end{pmatrix} y^{n(s-\half)} d^*z_2 \frac{dy}{y}.
\end{equation}
Then $\mathcal{L}(s, F \times \overline{u_j}) = 0$ if $u_j$ is odd, while if $u_j$ is even, we have
\begin{equation}
 \mathcal{L}(s, F \times \overline{u_j}) = 2G_j(s) A_F(1) \overline{B_j}(1)L(s, F \times \overline{u_j}).
\end{equation}
\end{mylemma}
\begin{proof}
 We have the Fourier expansion
\begin{multline}
 F(z) = \sum_{\gamma \in U_n(\mathbb{Z}) \backslash \SL{n}{Z}} \sum_{m_1 \geq 1} \dots \sum_{m_{n-1} \geq 1} \sum_{m_n \neq 0} \frac{A_F(m_1, \dots, m_n)}{\prod_{k=1}^n |m_k|^{\frac{k(n+1-k)}{2}}} 
\\
W_J\left( \begin{pmatrix} m_{1} \dots |m_n| &  &  &  \\ & \ddots &  &  \\  &  & m_1 &  \\  &  &  & 1 \end{pmatrix} \cdot \begin{pmatrix} \gamma & \\ & 1 \end{pmatrix} z, i\alpha, \psi_{1,\dots, 1, \frac{m_n}{|m_n|}} \right).
\end{multline}
Here $i\alpha = (i\alpha_1, \dots, i\alpha_n)$ are the Langlands parameters of $F$.
Then writing $M$ as shorthand for the diagonal matrix above, we have by unfolding
\begin{multline}
\label{eq:ujFinnerproduct}
 \int_{\SL{n}{Z} \backslash \mathcal{H}^n} \overline{u_j}(z_2)  F\begin{pmatrix} z_2 y & \\ & 1 \end{pmatrix}  d^*z_2
=
\sum_{m_1 \geq 1} \dots \sum_{m_{n-1} \geq 1} \sum_{m_n \neq 0} \frac{A_F(m_1, \dots, m_n)}{\prod_{k=1}^n |m_k|^{\frac{k(n+1-k)}{2}}} 
\\
\times \int_{U_n(\mathbb{Z}) \backslash \mathcal{H}^n} \overline{u_j}(z_2) W_J\Big(M \begin{pmatrix} z_2 y & \\ & 1 \end{pmatrix} , \psi_{1,\dots, 1, \frac{m_n}{|m_n|}} \Big) d^* z_2.
\end{multline}
By \cite{Goldfeld} p.132, and with $\psi_M(x) = e(m_2 x_{n-1, n} + m_3 x_{n-2, n-1} + \dots + m_n x_{1,2})$, we have
\begin{align}
 W_J\Big(M \begin{pmatrix} z_2 y & \\ & 1 \end{pmatrix} z, \psi_{1,\dots, 1, \frac{m_n}{|m_n|}} \Big) &= \psi_M(x)  W_J\Big(M \begin{pmatrix} Y y & \\ & 1 \end{pmatrix}, \psi_{1,\dots, 1, \frac{m_n}{|m_n|}} \Big)
\\
&= \psi_M(x)  W_J\Big(M \begin{pmatrix} Y y & \\ & 1 \end{pmatrix}, \psi_{1,\dots,  1} \Big).
\end{align}
Also note
\begin{multline}
 \int_0^1 \dots \int_0^1 \overline{u_j}(z_2) \psi_M(x) \prod_{1 \leq i < j \leq n} d x_{i,j} 
\\
= \frac{\overline{B_j}(m_2, \dots, m_n)}{\prod_{k=2}^n |m_k|^{\frac{(k-1)(n+1-k)}{2}}} \overline{W}_J \left( \begin{pmatrix} m_{2} \dots |m_n| y_2 \dots y_n &  &  &  \\ & \ddots &  &  \\  &  & m_2 y_2 &  \\  &  &  & 1 \end{pmatrix} , i\beta \right),
\end{multline}
where $i\beta = (i\beta_1, \dots, i\beta_n)$ are the Langlands parameters of $u_j$, and where here and in the following we do not write $\psi_{1,\dots,1}$ in the definition of the Jacquet Whittaker function.
Then we have
\begin{multline}
\label{eq:Lmiddlecalculation}
 \mathcal{L}(s, F \times \overline{u_j}) = 
\sum_{m_1 \geq 1} \dots \sum_{m_{n-1} \geq 1} \sum_{m_n \neq 0} \frac{A_F(m_1, \dots, m_n) \overline{B_j}(m_2, \dots, m_n)}{m_1^{n/2} \prod_{k=2}^n |m_k|^{\frac{(n+1-k)(2k-1)}{2}} }
\\
\int_0^{\infty} \dots \int_0^{\infty} W_J\Big(M \begin{pmatrix} Y y & \\ & 1 \end{pmatrix}, i\alpha \Big) 
\overline{W}_J \left( \begin{pmatrix} m_{2} \dots |m_n| y_2 \dots y_n &  &  &  \\ & \ddots &  &  \\  &  & m_2 y_2 &  \\  &  &  & 1 \end{pmatrix} ,i\beta \right)
\\
y_1^{n (s-\half)} \frac{dy_1}{y_1} \prod_{k=2}^n y_k^{-(k-1)(n+1-k)} \frac{dy_k}{y_k}.
\end{multline}
The inner integrals above simplify as
\begin{multline}
 \int_0^{\infty} \dots \int_0^{\infty} W_J\left(M \begin{pmatrix} y_1 y_2 \dots y_n (y_2^{n-1} \dots y_n)^{-1/n} &  &  &  \\ & \ddots &  &  \\  &  & y_1 (y_2^{n-1} \dots y_n)^{-1/n} &  \\  &  &  & 1 \end{pmatrix}, i\alpha  \right)
\\ 
\overline{W}_J \left( \begin{pmatrix} m_{2} \dots |m_n| y_2 \dots y_n &  &  &  \\ & \ddots &  &  \\  &  & m_2 y_2 &  \\  &  &  & 1 \end{pmatrix} ,i\beta \right)
y_1^{n (s-\half)} \frac{dy_1}{y_1} \prod_{k=2}^n y_k^{-(k-1)(n+1-k)} \frac{dy_k}{y_k}.
\end{multline}
Changing variables $y \rightarrow y_1 (y_2^{n-1} \dots y_n)^{1/n}$ gives that this is
\begin{multline}
 \int_0^{\infty} \dots \int_0^{\infty} (y_1^n y_2^{n-1} \dots y_n)^{(s-\half)} W_J \left(M \begin{pmatrix} y_1 y_2 \dots y_n  &  &  &  \\ & \ddots &  &  \\  &  & y_1 &  \\  &  &  & 1 \end{pmatrix}, i\alpha  \right) 
\\
\overline{W}_J \left( \begin{pmatrix} m_{2} \dots |m_n| y_2 \dots y_n &  &  &  \\ & \ddots &  &  \\  &  & m_2 y_2 &  \\  &  &  & 1 \end{pmatrix}, i\beta  \right)
 \frac{dy_1}{y_1} \prod_{k=2}^n y_k^{-(k-1)(n+1-k)} \frac{dy_k}{y_k}.
\end{multline}
Next we change variables with $y_j \rightarrow y_j/|m_j|$, getting
\begin{multline}
 \int_0^{\infty} \dots \int_0^{\infty} W_J\left(\begin{pmatrix} y_1 y_2 \dots y_n  &  &  &  \\ & \ddots &  &  \\  &  & y_1 &  \\  &  &  & 1 \end{pmatrix} , i\alpha \right)
\overline{W}_J \left( \begin{pmatrix} y_2 \dots y_n &  &  &  \\ & \ddots &  &  \\  &  &  y_2 &  \\  &  &  & 1 \end{pmatrix}, i\beta  \right)
\\
\Big(\prod_{k=1}^n \frac{1}{|m_k|^{(n+1-k)s}} \Big) \Big(\prod_{k=1}^n |m_k|^{\frac{(n+1-k)(2k-1)}{2}} \Big)
(y_1^n y_2^{n-1} \dots y_n)^{(s-\half)} \frac{dy_1}{y_1} \prod_{k=2}^n y_k^{-(k-1)(n+1-k)} \frac{dy_k}{y_k}.
\end{multline}
Inserting this into \eqref{eq:Lmiddlecalculation}, we obtain 
\begin{equation}
 \mathcal{L}(s, F \times \overline{u_j}) =  2\overline{B_j}(1) A_F(1) L(s, F \times \overline{u_j}) G_j(s),
\end{equation}
provided $u_j$ is even,
where
\begin{multline}
 G_j(s) = \int_0^{\infty} \dots \int_0^{\infty} W_J \left(\begin{pmatrix} y_1 y_2 \dots y_n  &  &  &  \\ & \ddots &  &  \\  &  & y_1 &  \\  &  &  & 1 \end{pmatrix}, i\alpha  \right)
\\
\overline{W}_J \left( \begin{pmatrix} y_2 \dots y_n &  &  &  \\ & \ddots &  &  \\  &  &  y_2 &  \\  &  &  & 1 \end{pmatrix}, i\beta  \right)
(y_1^n y_2^{n-1} \dots y_n)^{(s-\half)}  \prod_{k=1}^n y_k^{-(k-1)(n+1-k)} \frac{dy_k}{y_k}.
\end{multline}
Let
\begin{equation}
\label{eq:WhittakerScaledDef} 
W_J^*(y, i\beta) = W_J(y, i\beta) \prod_{j=1}^{n-1} \prod_{j \leq k \leq n-1} \pi^{-\half - iv_{j,k}} \Gamma(\thalf + iv_{j,k}),
\end{equation}
where
\begin{equation}
 iv_{j,k} = \frac{i}{2} \sum_{l=0}^{j-1} (\beta_{n-k+l} - \beta_{n-k+l+1}) = \frac{i}{2} (\beta_{n-k} - \beta_{n-k+j}), 
\end{equation}
Here $W_J^*(y, i\nu)$ 
is the Whittaker function normalized by Stade \cite{StadeAJM}.  We quote a formula of Stade \cite[Theorem 3.4]{StadeAJM}:
\begin{multline}
\int_0^{\infty} \dots \int_0^{\infty} W_J^* \left(\begin{pmatrix} y_1 y_2 \dots y_n  &  &  &  \\ & \ddots &  &  \\  &  & y_1 &  \\  &  &  & 1 \end{pmatrix}, i\alpha  \right)
\overline{W}_J^* \left( \begin{pmatrix} y_2 \dots y_n &  &  &  \\ & \ddots &  &  \\  &  &  y_2 &  \\  &  &  & 1 \end{pmatrix}, i\beta  \right)
\\
\times \prod_{k=1}^n (\pi y_k)^{(n+1-k)s} 2y_k^{-(n+1-k)(k-\half)} \frac{dy_k}{y_k}
= \prod_{l=1}^n \prod_{k=1}^{n+1} \Gamma(\frac{s+ i\beta_{l,j} - i\alpha_k}{2}).
\end{multline}
From this, we deduce
\begin{multline}
 G_j^*(s) = \int_0^{\infty} \dots \int_0^{\infty} W_J^* \left(\begin{pmatrix} y_1 y_2 \dots y_n  &  &  &  \\ & \ddots &  &  \\  &  & y_1 &  \\  &  &  & 1 \end{pmatrix}, i\alpha \right)
\\
\times \overline{W}_J^* \left( \begin{pmatrix} y_2 \dots y_n &  &  &  \\ & \ddots &  &  \\  &  &  y_2 &  \\  &  &  & 1 \end{pmatrix}, i\beta  \right)
 \prod_{k=1}^n y_k^{(n+1-k)(s+\half-k)} \frac{dy_k}{y_k}
\\
= 2^{-n} \pi^{-\frac{n(n+1)}{2} s} \prod_{l=1}^n \prod_{k=1}^{n+1} \Gamma(\frac{s+ i\beta_{l,j} - i\alpha_k}{2}),
\end{multline}
and hence
\begin{multline}
 G_j(s) = G_j^*(s) \prod_{j=1}^{n-1} \prod_{j \leq k \leq n-1} \pi^{\half - v_{j,k}} \Big(\prod_{1 \leq k < l \leq n} \Gamma(\frac{1-i\beta_k+i\beta_l}{2}) \Big)^{-1} 
\\
\prod_{j=1}^{n} \prod_{j \leq k \leq n} \pi^{\half + v_{j,k}^*} \Big(\prod_{1 \leq k < l \leq n+1} \Gamma(\frac{1+i\alpha_k-i\alpha_l}{2}) \Big)^{-1},
\end{multline}
which becomes
\begin{multline}
 G_j^*(s) \pi^{\frac{n^2}{2} - \sum_{j=1}^{n-1} \sum_{j \leq k \leq n-1} v_{j,k} + \sum_{j=1}^n \sum_{j \leq k \leq n} v_{j,k}^*} 
\\
\prod_{1 \leq k < l \leq n} \Gamma(\frac{1-i\beta_k+i\beta_l}{2})^{-1} 
\prod_{1 \leq k < l \leq n+1} \Gamma(\frac{1+i\alpha_k-i\alpha_l}{2})^{-1}. \qedhere
\end{multline}
\end{proof}

\begin{proof}[Proof of Proposition \ref{prop:normformula}]
 By the spectral decomposition of $\SL{n}{Z}$ \cite{MW},
\begin{equation}
 L^2(\SL{n}{Z} \backslash \mathcal{H}^n) = C_{\text{cusp}}(\SL{n}{Z} \backslash \mathcal{H}^n) \oplus (\text{Residual spectrum}) \oplus (\text{Continuous spectrum}).
\end{equation}
From the proof of Lemma \ref{lemma:ujFinnerproduct}, one sees that $\langle f_{y_1}, \phi \rangle = 0$ if $\phi$ has only degenerate Fourier coefficients.  As a result, the residual spectrum does not enter, and only the cuspidal Eisenstein series contribute (note there are more Eisenstein series in the continuous spectrum in general), i.e.,
\begin{multline}
 f_{y_1}(z_2) = 
\sum_{\substack{
\SL{n}{Z} \\ \text{ cuspidal spectrum}}} \langle f_{y_1}, u_j \rangle u_j(z_2) 
\\
+ 
\sum_{k=1}^{r}
\sum_{\substack{\SL{n_k}{Z} \\ \text{ cuspidal spectrum}}}  
c_{n_1,\dots,n_r} \int_{\mr^r} \langle f_{y_1}, E_{P_{n_1, \dots, n_r}}(\cdot, iv, \phi_j) \rangle E_{P_{n_1, \dots, n_r}}(z_2, iv, \phi_j)   dv_1^* \dots dv_{r-1}^*
\\
+ c \intR \dots \intR \langle f_{y_1}, E_{P_{\text{min}}}(\cdot, iw) \rangle  E_{P_{\text{min}}}(z_2, iw)d \beta_{1,w} \dots d \beta_{n-1, w},
\end{multline}
where $n_1 + \dots + n_r = n$ with $n_i \geq 1$, and $(n_1, \dots, n_r)$ runs through all such partitions of $n$.
By Parseval,
\begin{multline}
 \int_{\SL{n}{Z} \backslash \mathcal{H}^n} |f_{y_1}(z_2)|^2 d^* z_2 = 
\sum_{\substack{
\SL{n}{Z}  \\ \text{ cuspidal spectrum}}} |\langle f_{y_1}, u_j \rangle|^2
\\
+ 
\sum_{k=1}^{r}
\sum_{\substack{\SL{n_k}{Z} \\ \text{ cuspidal spectrum}}} 
c_{n_1, \dots, n_r} \int_{\mr^r} |\langle f_{y_1}, E_{P_{n_1, \dots, n_r}}(\cdot, iv, \phi_j) \rangle|^2   dv_1^* \dots dv_{r-1}^*
\\
+ c \intR \dots \intR |\langle f_{y_1}, E_{P_{\text{min}}}(\cdot, iw) \rangle |^2 d\beta_{1,w} \dots d\beta_{n-1,w}.
\end{multline}
The Plancherel formula says
\begin{equation}
 \int_0^{\infty} |h(y)|^2 \frac{dy}{y} = \frac{n}{2\pi} \intR |\widetilde{h}(n i t)|^2 dt, \qquad \widetilde{h}(ni t) = \int_0^{\infty} h(y) y^{ni t} \frac{dy}{y}.
\end{equation}
Hence
\begin{multline}
 \int_{-\infty}^{\infty} |\langle f_{y_1}, u_j \rangle|^2 \frac{dy_1}{y_1} =\frac{n}{2\pi} \int_{-\infty}^{\infty} \Big| \int_0^{\infty} \int_{\SL{n}{Z} \backslash \mathcal{H}^n} f_{y_1}(z_2) \overline{u_j}(z_2) d^* z_2 y_1^{n i t} \frac{dy_1}{y_1} \Big|^2 dt
\\
= \frac{n}{2\pi} \intR |\mathcal{L}(1/2 + it, F \times \overline{u_j})|^2 dt.
\end{multline}
The other terms with the Eisenstein series arrive in a similar way.
\end{proof}

\section{Rankin-Selberg calculations}
In this section we relate the $L^2$ norm of a Maass form $F$ to the Rankin-Selberg $L$-function $L(s, F \times \overline{F})$:
\begin{myprop}
\label{prop:RankinSelbergL2formula}
Suppose that $F$ is a tempered Hecke-Maass form for $\SL{n+1}{Z}$.  Then for some absolute constant $c(n) > 0$, we have
\begin{equation}
 \langle F, F \rangle = c(n) |A_F(1)|^2 \text{Res}_{s=1} L(s, F \times \overline{F}).
\end{equation}
\end{myprop}
Xiannan Li \cite{XiannanLi} has shown that $\text{Res}_{s=1} L(s, F \times \overline{F}) \ll \lambda(F)^{\varepsilon}$ which shows $|A_F(1)|^2 \gg \lambda(F)^{-\varepsilon}$ provided $F$ is $L^2$-normalized.  The lower bound $\text{Res}_{s=1} L(s, F \times \overline{F}) \gg \lambda(F)^{-\varepsilon}$ is not known in general but would follow from the generalized Riemann hypothesis.

\begin{proof}
 We generalize the calculation by starting with $F$ and $G$ tempered Hecke-Maass forms on $\SL{n+1}{Z}$ with respective Langlands parameters $i\alpha = (i \alpha_1, \dots, i \alpha_{n+1})$, and $i\beta = (i\beta_1, \dots, i \beta_{n+1})$.  Recall the definition \eqref{eq:WhittakerScaledDef}.  We also quote the following formula of Stade \cite{StadeIsrael}.
\begin{multline}
 \Gamma(\frac{(n+1)s}{2}) \int_0^{\infty} \dots \int_0^{\infty} W_J^* (y, i\alpha) 
\overline{W}_J^*(y, i\beta) 
\prod_{j=1}^{n} (\pi y_j)^{(n+1-j)s} (2y_j)^{-(n+1-j)j} \frac{dy_j}{y_j}
\\
= \prod_{j=1}^{n+1} \prod_{k=1}^{n} \Gamma(\frac{s + i\alpha_j - i \beta_k}{2}),
\end{multline}
where $y = \text{diag}(y_1 \dots y_n, y_1 \dots y_{n-1}, \dots, y_1, 1)$, and in the calculation we have used the fact that the $\beta_j$ and $\alpha_k$ are real.
By a calculation on p.369 of \cite{Goldfeld}, we have if $F \overline{G}$ is even (that is, $F$ and $G$ are both even or both odd), then
\begin{equation}
\label{eq:FGinnerproduct}
 \zeta((n+1)s) \langle F \overline{G}, E_P(\cdot, \overline{s}) \rangle = 2 A_F(1) \overline{A_G(1)} L(s, F \times \overline{G}) G_{i\alpha, i\beta}(s),
\end{equation}
where
\begin{equation} G_{i\alpha, i\beta}(s) = \int_0^{\infty} \dots \int_0^{\infty} W_J(y, i\alpha) \overline{W}_J(y, i \beta) \det(y)^s d^*y,
\end{equation}
where $d^*y = \prod_{k=1}^n y_k^{-k(n+1-k)} \frac{dy_k}{y_k}$ and $\det(y) = \prod_{j=1}^n y_j^{n+1-j}$.  Thus
\begin{equation}
 \det(y)^s d^*y= \prod_{j=1}^n y_j^{(n+1-j)s} y_j^{-j(n+1-j)} \frac{dy_j}{y_j}.
\end{equation}
By Stade's formula, we have with $a=-\frac{n(n+1)}{2}$ and $b=\frac{n(n+1)(n+2)}{6}$
\begin{multline}
\label{eq:StadeGLnGLn}
 G_{i\alpha, i\beta}(s) =  \frac{\pi^{as} 2^b}{\Gamma(\frac{(n+1)s}{2})}
\Big[
\prod_{j=1}^n \prod_{j \leq k \leq n} \pi^{-\half - \half(i\alpha_{n+1-k}-i\alpha_{n+1-k+j})} \Gamma(\frac{1 + i\alpha_{n+1-k} - i\alpha_{n+1-k+j}}{2}) \Big]^{-1}
\\
\Big[\prod_{j=1}^n \prod_{j \leq k \leq n} \pi^{-\half + \half(i\beta_{n+1-k}-i\beta_{n+1-k+j})} \overline{\Gamma}(\frac{1 + i\beta_{n+1-k} - i\beta_{n+1-k+j}}{2})
\Big]^{-1}
\prod_{j=1}^{n+1} \prod_{k=1}^{n+1} \Gamma(\frac{s + i\alpha_j - i \beta_k}{2}).
\end{multline}
By \cite{Goldfeld} Proposition 10.7.5, 
$E_P^*(z,s) = \pi^{-(n+1)s/2} \Gamma((n+1)s/2) \zeta((n+1)s) E_P(z,s)$ has a simple pole at $s=1$ with say residue $R$.  Taking $F=G$ and the residue at $s=1$ on both sides of \eqref{eq:FGinnerproduct}, we have
\begin{multline}
 R \frac{\pi^{\frac{n+1}{2}}}{ \Gamma(\frac{n+1}{2})} \langle F, F \rangle = |A_F(1)|^2 \text{Res}_{s=1} L(s, F \times \overline{F})  \frac{\pi^a 2^b}{\Gamma(\frac{n+1}{2})} 
\prod_{j=1}^{n+1} \prod_{k=1}^{n+1} \Gamma(\frac{1 + i\alpha_j - i \alpha_k}{2})
\\
\times \Big[ \prod_{j=1}^n \prod_{j \leq k \leq n} \pi^{-\half - \half(i\alpha_{n+1-k}-i\alpha_{n+1-k+j})} \Gamma(\frac{1 + i\alpha_{n+1-k} - i\alpha_{n+1-k+j}}{2}) \Big]^{-1}
\\
\times \Big[\prod_{j=1}^n \prod_{j \leq k \leq n} \pi^{-\half + \half(i\alpha_{n+1-k}-i\alpha_{n+1-k+j})} \overline{\Gamma}(\frac{1 + i\alpha_{n+1-k} - i\alpha_{n+1-k+j}}{2})
\Big]^{-1}.
\end{multline}
Observe that the gamma factors involving $\alpha$ are cancelled, as are the powers of $\pi$ involving $\alpha$, and the proof is complete.
\end{proof}

\section{Local Weyl Law}
\label{section:Weyl}
The formula for $N(F)$ given by Proposition \ref{prop:normformula} involves a spectral sum of $\SL{n}{Z}$ Maass forms and as such we need some control on this spectral sum.  This topic has seen some major recent advances but the precise results required here do not yet exist. 

Suppose that $(i\beta_1, \dots, i\beta_{n})$ are the Langlands parameters of a Maass form on $GL_n$, with $\beta_1 + \dots + \beta_n = 0$.  We shall suppose that all forms are tempered so that $\beta_l \in  \mathbb{R}$, for all $1 \leq l \leq n$.  For a vector $\lambda = (\lambda_1, \dots, \lambda_n)$ with $\lambda_1 + \dots + \lambda_n = 0$, each $\lambda_l \in  \mathbb{R}$, set
\begin{equation}
\label{eq:mudefinition}
 \mu(\lambda) = \prod_{1\leq k < l \leq n} (1 + |\lambda_k - \lambda_l|).
\end{equation}
Here our $\mu(\lambda)$ is $\tilde{\beta}(\lambda)$ as given by (3.4) in \cite{LapidMuller}.   Then according to Proposition 4.5 of \cite{LapidMuller}, we have\footnote{Technically, Lapid and M\"{u}ller require a congruence subgroup of $\SL{n}{Z}$ so strictly speaking this result is not unconditional, but this is apparently a minor technical issue.}
\begin{equation}
 \#\{\beta : \| \beta - \lambda \| \leq 1 \} \ll \mu(\lambda),
\end{equation}
where the count is over Maass forms with Langlands parameter $i\beta$.  The corresponding lower bound is apparently not known in general.  However, the lower bound is known on average in the sense that the number of Maass forms with Langlands parameter $\beta$ lying in a region of the form $t \Omega$ is
\begin{equation}
 \asymp \int_{t \Omega} \mu(\lambda) d \lambda,
\end{equation}
as $t \rightarrow \infty$.  In fact, \cite{LapidMuller} find an asymptotic count with a power saving.  

For our applications here, we find it most desirable to have the following  estimates.  Let $\lambda$ be given.  Then for some fixed absolute constant $K \geq 1$,
\begin{equation}
\label{eq:WLWLlower}
 \mu(\lambda) \ll \sideset{}{^+}\sum_{\beta: \| \beta - \lambda \| \leq K} |B_{\beta}(1)|^2,
\end{equation}
where $B_{\beta}(1)$ is the first Fourier coefficient of the Hecke-Maass form associated to $\beta$, and the $+$ denotes the sum is restricted to even forms.  We also require an analogous upper bound with the continuous spectrum included, namely
\begin{multline}
 \label{eq:WLWLupper}
\sum_{\beta: \| \beta - \lambda \| \leq 1} |B_{\beta}(1)|^2 + 
\sum_{\phi_{j_1}} \dots \sum_{\phi_{j_r}} c_{n_1, \dots, n_r} \int_{\| \beta_{j_1, \dots, j_r} - \lambda \| \leq 1} |B_{P_{n_1, \dots, n_r}}(1)|^2 dv_1^* \dots dv_{r-1}^*
\\
+ c \int_{\| \beta_w - \lambda \| \leq 1} |B_{P_{min}}(1)|^2 d_{\beta_{1,w}} \dots d \beta_{n-1, w} \ll \mu(\lambda),
\end{multline}
where $i\beta_{j_1, \dots j_r}$ is the vector of Langlands parameters of $E_{P_{n_1, \dots, n_r}}$ and $i \beta_w$ is the vector of Langlands parameters of $E_{min}$, all the other notations being defined before and in Proposition \ref{prop:normformula}.
We call \eqref{eq:WLWLlower}-\eqref{eq:WLWLupper} the {\em weighted local Weyl law} (though technically it is not an asymptotic).  

Blomer has recently shown \eqref{eq:WLWLlower}-\eqref{eq:WLWLupper} for $n=3$ \cite{Blomer}. 
Blomer's proof relies on the $GL_3$ Kuznetsov formula as opposed to the Arthur-Selberg trace formula.  The weighting inherent in the Kuznetsov formula (with the first Fourier coefficient as weights) is much more natural in our application here.

\section{Archimedean development of the norm formula}
\label{section:Archimedean}
In this section we scrutinize the Archimedean part of the norm formula.  We are eventually led to a combinatorial-type problem of integrating a certain function over a polytope.

According to Proposition \ref{prop:normformula}, write $N(F) = N_d(F) + N_{\text{max}}(F) + N_{\text{min}}(F)$.  Then $N(F) \geq N_d(F)$ so for the purpose of the lower bound in Theorem \ref{thm:normLB} we may restrict our attention to $N_d(F)$.  Our methods for obtaining an upper bound on $N_d(F)$ turn out to apply to the Eisenstein series also.  It is a familiar fact from $\SL{2}{Z}$ that the continuous spectrum is usually negligible compared to the cuspidal spectrum.

By Proposition \ref{prop:normformula}, we have (switching $u_j$ with $\overline{u_j}$)
\begin{equation}
\label{eq:normUBspectraldecomposition}
N_d(F) \asymp |A_F(1)|^2 \sumstar_{j} |B_j(1)|^2 \intR |L(1/2 + it, F \times u_j)|^2 q(t, \alpha, \beta_j) dt,
\end{equation}
where the $*$ indicates the sum is restricted to cusp form $u_j$ of the same parity of $F$, the implied constants depend only on $n$, and
\begin{equation}
\label{eq:qtalphabetaDEF}
q(t, \alpha, \beta_j) = \frac{\prod_{l=1}^{n} \prod_{k=1}^{n+1} |\Gamma(\frac{1/2 + it+ i\alpha_k + i\beta_{l,j}}{2}) |^2}{\Big(\prod_{1 \leq k < l \le n+1} |\Gamma(\frac{1+ i\alpha_k - i\alpha_l}{2}) |^2\Big) \Big(\prod_{1 \leq k < l \le n} | \Gamma(\frac{1+ i\beta_{k,j} - i\beta_{l,j}}{2})|^2 \Big)}.
\end{equation}
Recall that we are assuming all our forms are tempered so that $\alpha_k, \beta_{l,j} \in  \mathbb{R}$.  Then Stirling's formula shows that
\begin{equation}
\label{eq:qtalphabetaSTIRLING}
q(t, \alpha, \beta) \asymp \exp(- \frac{\pi}{2} r(t, \alpha, \beta)) \prod_{l=1}^{n} \prod_{k=1}^{n+1} (1 + |t + \alpha_k + \beta_{l}|)^{-1/2},
\end{equation}
where
\begin{equation}
r(t, \alpha, \beta) = \sum_{l=1}^n \sum_{k=1}^{n+1} |t + \alpha_k + \beta_{l}| - \sum_{1 \leq k < l \leq n+1} |\alpha_k-\alpha_l| - \sum_{1 \leq k < l \leq n} |\beta_{k} - \beta_{l}| .
\end{equation}

Since $F$ is a nice function that is $L^2$-normalized, all of its $L^p$ norms (as well as $N(F)$) are polynomially bounded in terms of $\lambda(F)$.  However, this is not (yet!) clear from the formula \eqref{eq:normUBspectraldecomposition}.  In Lemma \ref{lemma:rnonnegative} below we will show that $r(t, \alpha, \beta) \geq 0$ for all $t \in \mathbb{R}$, $\alpha \in \mathbb{R}^{n+1}$, and $\beta \in \mathbb{R}^n$ so that at least $q(t,\alpha, \beta)$ is not exponentially large.  We will also show that for some nice set of $t$ and $\beta$ that $r(t, \alpha, \beta) = 0$; outside of this set, $r(t,\alpha,\beta)$ quickly becomes large which allows us to finitize the integral over $t$ and sum over $j$ in \eqref{eq:normUBspectraldecomposition}.  The set of $\beta$'s such that there exists a $t$ with $r(t, \alpha, \beta) = 0$ turns out to define a polytope that we shall study extensively in Section \ref{section:polytope}.

\begin{mylemma}
\label{lemma:rnonnegative}
 Suppose $\alpha \in \mr^{n+1}$, $\beta \in \mr^n$, and $t \in \mathbb{R}$.  Then $r(t,\alpha, \beta) \geq 0$.  
\end{mylemma}

\begin{proof}
Note that as a function of $t$, $r(t, \alpha, \beta)$ is piecewise linear.  It has slope $n(n+1)$ for large $t > 0$, and slope $-n(n+1)$ for large $-t > 0$.  Each time $-t$ passes through a point $\alpha_k + \beta_{l}$ the slope changes by $2$.  By this reasoning we see that the graph of $r(t,\alpha, \beta)$ is ``flat" (has zero slope) on an interval between the two ``middle" points $\alpha_k + \beta_{l}$.  To this end, it is natural to partition the set $S=\{ \alpha_k + \beta_l : 1 \leq k \leq n+1, 1 \leq l \leq n \}$ as $S_{+} \cup S_{-}$ where $|S_{+}| = |S_{-}| = \half |S|$, and each element of $S_+$ is $\geq$ each element of $S_-$; in case of multiplicity there may be more than one way to choose $S_+$ and $S_-$.  
Define the {\em median interval} $I_M$ as
\begin{equation}
\label{eq:IMdef}
 I_M = \{t \in \mathbb{R} : t+ s_{+} \geq 0 \text{ and } t + s_{-} \leq 0 \text{ for all } s_{+} \in S_{+} \text{ and } s_{-} \in S_{-} \}.
\end{equation}
Note that $I_M$ may consist of only one point if $S_+ \cap S_-$ is nonempty.

By elementary reasoning, the minimum of $r(t, \alpha, \beta)$ must occur when $t =- \alpha_k - \beta_{l}$ for some $k,l$.  By symmetry, say it occurs at $-\alpha_{n+1} -\beta_{n}$.
Then
\begin{multline}
\label{eq:rF}
r(-\alpha_{n+1} -  \beta_n, \alpha, \beta) = \sum_{l=1}^n \sum_{k=1}^{n+1} | \alpha_k - \alpha_{n+1} + \beta_{l} - \beta_n| 
- \sum_{1 \leq k < l \leq n+1} |\alpha_k-\alpha_l| 
- \sum_{1 \leq k < l \leq n} |\beta_{k} - \beta_{l}| .
\end{multline}
Proceed by induction.
If $n=1$ then \eqref{eq:rF} is zero, and we are done.  Suppose $n > 1$.
In the first sum above, take $k=n+1$ and $l=n$ separately, and similarly take $l=n+1$ and $l=n$ separately in the second and third sums above.  Then we obtain
\begin{equation}
\label{eq:rF2}
r(-\alpha_{n+1} -  \beta_n, \alpha, \beta) = \sum_{l=1}^{n-1} \sum_{k=1}^{n} | - \alpha_{n+1}  - \beta_n +\alpha_k + \beta_{l} | 
- \sum_{1 \leq k < l \leq n} |\alpha_k-\alpha_l| 
- \sum_{1 \leq k < l \leq n-1} |\beta_{k} - \beta_{l}|.
\end{equation} 
The right hand side of \eqref{eq:rF2} takes the form $r(t, \alpha', \beta')$ for some $t \in \mathbb{R}$ (in fact $t=-\alpha_{n+1}-\beta_n$), $\alpha' = (\alpha_1, \dots, \alpha_{n})$, and $\beta' = (\beta_1, \dots, \beta_{n-1})$,  so by the induction hypothesis we are done.
\end{proof}

\begin{mylemma}
\label{lemma:rzero}
 Suppose that $\alpha = (\alpha_1, \dots, \alpha_{n+1}) \in \mr^{n+1}$, $\beta = (\beta_1, \dots, \beta_n) \in \mr^n$ and $\alpha_1 \geq \alpha_2 \geq \dots \geq \alpha_{n+1}$, $\beta_1 \geq \beta_2 \geq \dots \geq \beta_n$. 
Then there exists $t\in \mr$ such that $r(t,\alpha,\beta) = 0$ if and only if
\begin{equation}
\label{eq:alphajbetakcondition}
\alpha_{n+1-k} + \beta_{k} \geq \alpha_{n+2-l} + \beta_{l}, 
\end{equation}
for any $k,l \in \{1, \dots, n\}$.  
\end{mylemma}
It may be helpful to visualize the numbers $\alpha_k + \beta_l$ in the following array:
\begin{equation}
\label{eq:alphabetaarray}
 \begin{array}{cccc}
  \alpha_1 + \beta_1 & \alpha_1 + \beta_2 & \dots & \alpha_1 + \beta_n \\
\cline{4-4}
\alpha_2 + \beta_1 & \alpha_2 + \beta_2 & \dots & \alpha_2 + \beta_n \\
\vdots & \ddots & \dots & \dots \\
\cline{2-2}
\alpha_n + \beta_1 & \alpha_n + \beta_2 & \dots & \alpha_n + \beta_n \\
\cline{1-1}
\alpha_{n+1} + \beta_1 & \alpha_{n+1} + \beta_2 & \dots & \alpha_{n+1} + \beta_n.
 \end{array}
\end{equation}
Notice that the entries of the array are decreasing (that is, non-increasing) in each row and in each column.  The condition \eqref{eq:alphajbetakcondition} means that each entry above one of the horizontal lines is $\geq$ each entry below one of the horizontal lines.
\begin{proof}
 From the orderings imposed on the $\alpha_k$ and $\beta_l$, we have
\begin{equation}
 r(t, \alpha, \beta) = \sum_{l=1}^n \sum_{k=1}^{n+1} |t + \alpha_k + \beta_{l}| - n \alpha_1 - (n-2) \alpha_2 - \dots + n \alpha_{n+1} - (n-1)\beta_1  - \dots + (n-1)\beta_n.
\end{equation}
From Lemma \ref{lemma:rnonnegative}, we know $r(t,\alpha, \beta) \geq 0$.  Furthermore, as noted in the proof of Lemma \ref{lemma:rnonnegative}, $r$ is minimized when $t \in I_M$, so the only possible region of $t$ with $r(t,\alpha, \beta) = 0$ is $t \in I_M$.  By glancing at \eqref{eq:alphabetaarray}, one can see that $S_{+}$ consists of the elements $\alpha_k + \beta_l$ lying above the horizontal lines, and similarly the elements of $S_{-}$ are below the horizontal lines.  Using this, one calculates easily that $r(t,\alpha,\beta) = 0$ for such $t$.  Thus, if \eqref{eq:alphajbetakcondition} holds then $r(t, \alpha, \beta) = 0$ for $t \in I_M$.

Next we show the other half of the ``iff'' statement.
In general (not necessarily assuming \eqref{eq:alphajbetakcondition} holds), one obtains that for $t \in I_M$ that
\begin{equation}
\label{eq:rtalphabetageneral}
 r(t,\alpha,\beta) = \sum_{s_+ \in S_{+}} s_+ - \sum_{s_- \in S_-} s_- - n \alpha_1 - (n-2) \alpha_2 - \dots + n \alpha_{n+1} - (n-1)\beta_1  - \dots + (n-1)\beta_n.
\end{equation}
Let $T_+$ be the set of elements $\alpha_k + \beta_l$ above the horizontal lines in \eqref{eq:alphabetaarray}, and similarly $T_-$ is the set of elements below the horizontal lines.  Then our previous calculation shows
\begin{equation}
 \sum_{t_+ \in T_+} t_+ - \sum_{t_- \in T_-} t_- - n \alpha_1 - (n-2) \alpha_2 - \dots + n \alpha_{n+1} - (n-1)\beta_1  - \dots + (n-1)\beta_n = 0,
\end{equation}
so inserting this into \eqref{eq:rtalphabetageneral}, we obtain
\begin{equation}
 r(t, \alpha, \beta) = \sum_{s_+ \in S_{+}} s_+ -\sum_{t_+ \in T_+} t_+ + \sum_{t_- \in T_-} t_- - \sum_{s_- \in S_-} s_-  = 2 \sum_{s_+ \in S_+ \cap T_-} s_+ - 2 \sum_{s_- \in S_- \cap T_+} s_-.
\end{equation}
If \eqref{eq:alphajbetakcondition} does not hold then there exists an $s_+ \in S_+ \cap T_-$ and $s_- \in S_- \cap T_+$ such that $s_+ > s_-$, and hence $r(t,\alpha, \beta) > 0$.
\end{proof}

The structure of the set of $t \in I_M$ and $\beta$ satisying \eqref{eq:alphajbetakcondition} is related to the branching law\footnote{A branching law describes how a representation of a group $G$ decomposes into irreducible representations upon restriction to a subgroup $H$ of $G$.} of $\GL{n+1}{C}$ to $\GL{n}{C}$ as we now explain.  Let $\lambda_j = \beta_j + t$ for all $j$; in terms of $\lambda = (\lambda_1, \dots, \lambda_n)$, the condition that $t \in I_M$ and $\beta$ satisfies \eqref{eq:alphajbetakcondition} is equivalent to $\lambda_j + \alpha_{n+1-j} \geq 0$ and $\lambda_j + \alpha_{n+2-j} \leq 0$, for all $j \in \{1,\dots, n\}$.  This is in turn equivalent to
\begin{equation}
\label{eq:interlacing}
 -\alpha_{n+1} \geq \lambda_1 \geq -\alpha_n \geq \lambda_2 \geq \dots \geq -\alpha_2 \geq \lambda_n \geq -\alpha_1.
\end{equation}
Note that $\widetilde{\alpha} :=(-\alpha_{n+1}, \dots, -\alpha_1)$ are the Langlands parameters of $\overline{F}$, and \eqref{eq:interlacing} then says that $\lambda$ {\em interlaces} $\widetilde{\alpha}$ (it would also be natural to take the dual of $u_j$ instead of $F$).  
Theorem 8.1.1 of \cite{GoodmanWallach}, for example, expresses the branching law from $\GL{n+1}{C}$ to $\GL{n}{C}$ via the interlacing of the highest weight vectors of the corresponding irreducible representations of the two groups.
For more information about branching laws, for example see Chapter 8 of \cite{GoodmanWallach} or Chapter XVIII of \cite{Zelobenko}.


\section{The upper bound}
\label{section:upperbound}
Lemmas \ref{lemma:rnonnegative} and \ref{lemma:rzero} give us good control on the exponential part of $q(t,\alpha, \beta)$.  We also need to understand the rational part of $q$ which is established with the following.
\begin{mylemma}
\label{lemma:qtalphabetaUB}
 Suppose that $\alpha, \beta$ are as in Lemma \ref{lemma:rzero}, \eqref{eq:alphajbetakcondition} holds, and $t \in I_M$.  Then
\begin{equation}
\label{eq:qtalphabetaUB}
 q(t, \alpha, \beta) \ll \frac{1}{\mu(\beta)} \prod_{k+l = n+1} (1+ |t+ \alpha_k + \beta_l|)^{-1/2}  \prod_{k+l = n+2} (1+ |t+ \alpha_k + \beta_l|)^{-1/2}.
\end{equation}
\end{mylemma}
Recall that $\mu(\beta)$ is defined by \eqref{eq:mudefinition}.
\begin{proof}
We estimate $q(t,\alpha,\beta)$ using \eqref{eq:qtalphabetaSTIRLING}.
 Since we assume $t \in I_M$ and \eqref{eq:alphajbetakcondition} holds, we have $r(t,\alpha,\beta) = 0$.  The terms with $k+l = n+1$ and $k+l=n+2$ are already present in \eqref{eq:qtalphabetaUB}; these are the terms corresponding to elements of the array \eqref{eq:alphabetaarray} directly above or below one of the horizontal lines.  Consider first the other terms above the horizontal lines in \eqref{eq:alphabetaarray}, that is, $1 + |t+ \alpha_k + \beta_l|$ with $k +l \leq n$.  Since $t \in I_M$ and we have the ordering \eqref{eq:alphajbetakcondition}, then
\begin{equation}
1 + |t+\alpha_k + \beta_l| = 1 + t + \alpha_k + \beta_l = 1 + (\beta_l - \beta_{n+1-k}) + (t + \alpha_{k} + \beta_{n+1-k}) \geq 1 + (\beta_l - \beta_{n+1-k}).
\end{equation}
Thus we have
\begin{equation}
 \prod_{k+l \leq n} (1 + |t + \alpha_k + \beta_l|)^{-1/2} \leq \prod_{k+l \leq n} (1 + |\beta_l - \beta_{n+1-k}|)^{-1/2} = \mu(\beta)^{-1/2},
\end{equation}
recalling the definition \eqref{eq:mudefinition}.
Similarly, for the terms with $k+ l \geq n+3$ we have
\begin{equation}
 1 + |t+ \alpha_k + \beta_l| = 1-t - \alpha_k - \beta_l = 1 - t - \alpha_k -\beta_{n+2-k} + (\beta_{n+2-k} - \beta_l) \geq 1 + (\beta_{n+2-k} - \beta_l).
\end{equation}
Then these terms with $k+l \geq n+3$ also contribute $\leq \mu(\beta)^{-1/2}$ to $q(t,\alpha,\beta)$, and the proof is complete.
\end{proof}

We shall need the following elementary integral bound.
\begin{mylemma}  Suppose $a, b \in \mathbb{R}$.  Then
\begin{equation}
\label{eq:telementaryintegral}
 \int_{-X}^X (1 + |t+a|)^{-1/2} (1+ |t+b|)^{-1/2} dt \leq \log(1+|a|+X) + \log(1+|b|+X).
\end{equation}
Similarly,
\begin{equation}
\label{eq:nelementarysum}
 \sum_{|n| \leq X} (1 + |n+a|)^{-1/2} (1 + |n+b|)^{-1/2} \ll \log(2+|a|+X) + \log(2+|b|+X)
\end{equation}
Finally, if $a < b$ then
\begin{equation}
\label{eq:hahb}
\int_{-b}^{-a} (1 + |t+a|)^{-1/2} (1+ |t+b|)^{-1/2} dt \leq 4.
\end{equation}
\end{mylemma}
 \begin{proof}
We begin with \eqref{eq:telementaryintegral}. Using $2\sqrt{xy} \leq x + y$, it is enough to consider the case $a=b$.  For $X \geq |a|$, we have
\begin{equation}
 \int_{-X}^{X} (1+ |t+a|)^{-1} dt = \log(1 + a + X) + \log(1 + X-a) \leq 2 \log(1 + |a| + X).
\end{equation}
One can check the same bound holds for $X < |a|$ also, so the proof of \eqref{eq:telementaryintegral} is complete.  The proof of \eqref{eq:nelementarysum} follows similar lines.

Let $h_u(t) = (1 + |t+u|)^{-1/2}$.  Then 
\begin{multline}
 \int_{-b}^{-a} h_a(t) h_b(t) dt = 2 \int_{-b}^{-b + \frac{b-a}{2}}  h_a(t) h_b(t) dt \leq 2(1 + \frac{b-a}{2})^{-1/2} \int_{-b}^{-b + \frac{b-a}{2}} h_b(t) dt
\\
= 4(1 + \frac{b-a}{2})^{-1/2} \big((1 + \frac{b-a}{2})^{1/2} - 1 \big) \leq 4. \qedhere
\end{multline}

\end{proof}

Now we are ready to prove Theorem \ref{thm:normUB}.  First we show the following
\begin{mylemma}
\label{lemma:qUB}
 We have
\begin{equation}
\label{eq:localsum}
 \sum_{\beta} |B_{\beta}(1)|^2 \int_{t \in I_M} q(t,\alpha,\beta) dt \ll 1,
\end{equation}
where the sum is over $\beta$, the Langlands parameters of the $SL_{n}(\mz)$ cuspidal spectrum, which also satisfy \eqref{eq:alphajbetakcondition}.
\end{mylemma}
Note that if $t \in I_M$, then it is implicit that \eqref{eq:alphajbetakcondition} holds, by Lemma \ref{lemma:rzero}.
\begin{proof}
Let $S$ denote the left hand side of \eqref{eq:localsum}.  
By Lemma \ref{lemma:qtalphabetaUB}, we have
\begin{equation}
S \ll \sum_{\beta} \int_{t \in I_M} |B_\beta(1)|^2 \frac{1}{\mu(\beta)} f(\beta + t) dt,
\end{equation}
where with $\lambda = (\lambda_1, \dots, \lambda_n)$, we set
\begin{equation}
 f(\lambda) = \prod_{k+l = n+1} (1+ |\alpha_k + \lambda_l|)^{-1/2}  \prod_{k+l = n+2} (1+ |\alpha_k + \lambda_l|)^{-1/2}.
\end{equation}
Let $\gamma = (\gamma_1, \dots, \gamma_n) \in \mathbb{Z}^n$ and let $H$ denote the hyperplane $u_1 + \dots + u_n = 0$.  Note that $f(\lambda') \asymp f(\lambda)$ if $\| \lambda - \lambda' \| = O(1)$.  Hence
\begin{equation}
\label{eq:Sbound}
 S \ll \sum_{\gamma \in \mathbb{Z}^n \cap H} \sum_{\beta: \|\beta - \gamma \| \leq K} \int_{t \in I_M} \frac{|B_{\beta}(1)|^2}{\mu(\gamma)} f(\gamma + t) dt.
\end{equation}
Recall that the condition $t \in I_M$ is equivalent to
\eqref{eq:interlacing}, that is,
\begin{equation}
 -\alpha_{n+1} \geq \beta_1 + t \geq -\alpha_n \geq \dots \geq -\alpha_2 \geq \beta_n + t \geq -\alpha_1,
\end{equation}
so by positivity we can extend this condition $t \in I_M$ to
\begin{equation}
\label{eq:tinterlacing}
 -\alpha_{n+2-k} + K \geq \gamma_k + t \geq -\alpha_{n+1-k} - K,
\end{equation}
for all $k\in \{1,2,\dots, n\}$.  Hence we obtain by \eqref{eq:WLWLupper} that
\begin{equation}
 S \ll \sum_{\gamma \in \mathbb{Z}^n \cap H} \int_{t} f(\gamma + t) dt ,
\end{equation}
where the integral is over $t$ such that \eqref{eq:tinterlacing} holds.
Next we argue that the sum over $\gamma$ can be replaced by an integral.  For this, we use the simple inequality
\begin{equation}
 \sum_{-a-K \leq m \leq -b + K} (1 + |m+a|)^{-1/2} (1 + |m+b|)^{-1/2} \ll 1 + \int_{-a-K}^{-b + K} (1 + |u+a|)^{-1/2} (1 + |u+b|)^{-1/2} du,
\end{equation}
and note that since $K \geq 1$, the integral is $\gg 1$, so in fact the sum can be bounded by a constant multiple of the integral.
Changing variables $\lambda_k = \gamma_k + t$, we have that the region of integration for $\lambda$ is a box.  Specifically, we have
\begin{equation}
\label{eq:Sbound2}
 S \ll \prod_{k=1}^{n} \int_{-\alpha_{n+1-k} -K}^{-\alpha_{n+2-k} +K} (1 + |\alpha_{n+2-k} + \lambda_k|)^{-1/2} (1 + |\alpha_{n+1-k} + \lambda_k|)^{-1/2} d\lambda_k.
\end{equation}
By \eqref{eq:hahb}, we deduce $S \ll 1$, as desired.
\end{proof}

\begin{mytheo}
\label{thm:normcuspUB}
 Assume the conditions of Theorem \ref{thm:normUB}.  Then
\begin{equation}
 N_d(F) \ll \lambda(F)^{\varepsilon}.
\end{equation}
\end{mytheo}
\begin{proof}
 This will follow from the proof of Lemma \ref{lemma:qUB} after some reductions.  The first thing to note is that since the spectral measure $\mu(\beta)$ is invariant under permutations of $(\beta_1, \dots, \beta_n)$, it suffices to consider the ordering as in Lemma \ref{lemma:rzero} (indeed, the Langlands parameters $(\beta_1, \dots, \beta_n)$ are only defined up to permutation anyway).

Next we need to finitize the integral and sum above.  If $I_M = [a,b]$ then let $I_M^* = [a-\log^2 \lambda(F), b + \log^2 \lambda(F)]$. We may restrict $t$ to $I_M^*$ since $q(t, \alpha, \beta)$ is exponentially small otherwise. Similarly we can restrict the $\beta_j$'s so that $\beta_{n+1-k} - \beta_{n+1-j} \leq \alpha_j - \alpha_{k+1} + \log^2(\lambda(F))$.  The Lindel\"{o}f Hypothesis implies that the $L$-functions are bounded by $\lambda(F)^{\varepsilon}$.  Then following the arguments of Lemma \ref{lemma:qUB}, we see that these slight extensions of the integral and sums, and the use of Lindel\"{o}f only alters the final bound by $\ll \lambda(F)^\varepsilon$.
\end{proof}

\begin{myprop}
 Assume the conditions of Theorem \ref{thm:normUB}.  Then
\begin{equation}
N_{\text{max}}(F) + N_{\text{min}}(F) 
\ll \lambda(F)^{\varepsilon}.
\end{equation}
\end{myprop}
\begin{proof}
We have that the contribution of a $E_{P_{n_1, \dots, n_r}}$ to \eqref{eq:normUBspectraldecomposition} is of the form
\begin{equation}
\ll |A_F(1)|^2  \sum_{k=1}^{r} \sum_{\substack{SL_{n_k}(\mathbb{Z})} }  \int |B_{P_{n_1, \dots, n_r}}(1)|^2 |L(1/2+it, F \times E_{P_{n_1, \dots, n_r}})|^2 q_v(t, \alpha, \beta) dt dv_1^* \dots dv_{r-1}^*,
\end{equation}
where $q_v(t,\alpha, \beta)$ is given by \eqref{eq:qtalphabetaSTIRLING} but with
$\beta$ defined by \eqref{eq:LanglandsEisenstein}, that is,
\begin{equation}
\beta = (v_1^* + \beta_{j_1,1}, \dots, v_1^*+ \beta_{j_1, n_1} \big| v_2^* + \beta_{j_2, n_1 + 1}, \dots, v_2^* + \beta_{j_2, n_1 + n_2} \big| \dots). 
\end{equation}
By the same reasoning as in the proof of Theorem \ref{thm:normcuspUB}, a bound 
of $\lambda(F)^{\varepsilon}$ for 
\begin{equation}
  \sum_{k=1}^{r} \sum_{SL_{n_k}(\mathbb{Z})} \int_{\mr^{r-1}}  \int_{t \in I_M} |B_{P_{n_1, \dots, n_r}}(1)|^2  q_v(t, \alpha, \beta) dt dv_1^* \dots dv_{r-1}^*
\end{equation}
will carry over to $N_{max}(F)$.  The proof of Lemma \ref{lemma:qUB} carries over almost without change.

%
%

The case of the minimal Eisenstein series is the easiest of all and we omit it.
\end{proof}

\section{A polytope}
\label{section:polytope}
This section is self-contained and our notation may not agree with other sections of the paper.

Let $\alpha, \beta$ be as in Lemma \ref{lemma:rzero}, suppose $t \in I_M$, and suppose that \eqref{eq:alphajbetakcondition} holds.  We think of $\alpha$ as given (it comes from the Langlands parameters of $F$) and we wish to understand the set of $\beta$ and $t$ such that \eqref{eq:alphajbetakcondition} holds with $t \in I_M$.
Without some extra work, it is not obvious that there even exists such $\beta$.
Let $x_j = \beta_j - \beta_{j+1}$ and $y_j = \alpha_{n+1-j}-\alpha_{n+2-j}$ so $x_j, y_j \geq 0$.  Then the system of inequalities \eqref{eq:alphajbetakcondition} is equivalent to the following system
\begin{equation}
\label{eq:system}
y_{j+1} + \dots + y_{k-1} \leq x_j + \dots + x_{k-1} \leq y_j + \dots + y_k,
\end{equation}
for $1 \leq j < k \leq n$.  We use the standard convention that if $j+1 > k-1$ then the left hand side of \eqref{eq:system} denotes $0$.

\begin{figure}[h]
\input{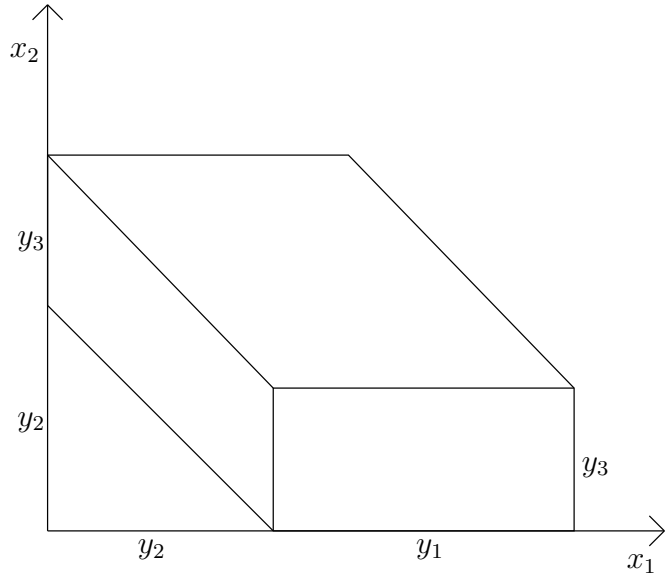}
\caption{The polytope $\mathcal{P}$ for $n=3$}
\label{fig:polytope}
\end{figure}
This defines a convex polytope in $\mathbb{R}^{n-1}$ which we denote as $\mathcal{P} = \mathcal{P}(y)$.  We suppose that none of the $y_j = 0$ since this is a somewhat degenerate situation. 
The description of $\mathcal{P}$ via \eqref{eq:system} is not conducive for our analysis.  Instead, we have a decomposition of $\mathcal{P}$ into $n$ parallelohedra each with a quite simple description.  Let $x = (x_1, \dots, x_{n-1})$, $w = (y_2, \dots, y_n)$, $v_1 = (1, 0, \dots, 0)$, $v_2 = (-1, 1, 0, \dots, 0)$, $v_3 = (0, -1, 1, 0, \dots, 0), \dots, v_{n-1} = (0, \dots, 0, -1, 1)$, $v_n = (0, \dots, 0, -1)$; so $v_2, \dots, v_{n-1}$ follow a pattern while the endpoint terms $v_1$ and $v_n$ have a different rule. 
\begin{mytheo}
\label{thm:Q=P}
 Define $\mathcal{Q}$ to be the convex polytope defined by the set of $x \in \mathbb{R}^{n-1}$ such that
\begin{equation}
\label{eq:polytopedecomposition}
  x = w +  t_1 y_1 v_1 + t_2 y_2 v_2 + \dots + t_n y_n v_n,
\end{equation}
for $0 \leq t_1, \dots, t_n \leq 1$.  Similarly, for each $j \in \{1, \dots, n \}$, let $\mathcal{Q}_j \subset \mathcal{Q}$ denote the parallelohedron defined by
\begin{equation}
 x = w + t_1 y_1 v_1 + \dots + t_{j-1} y_{j-1} v_{j-1} + t_{j+1} y_{j+1} v_{j+1} + \dots + t_n y_n v_n,
\end{equation}
for $0 \leq t_1, \dots, t_n \leq 1$.  Then $\mathcal{Q} = \mathcal{P}$.  Moreover, $\cup_{j=1}^{n} \mathcal{Q}_j = \mathcal{Q}$, and $\mathcal{Q}_j \cap \mathcal{Q}_k$ is a subset of an $(n-2)$-dimensional cone in $\mathbb{R}^{n-1}$ for $j \neq k$.
\end{mytheo}
Figure \ref{fig:polytope} illustrates Theorem \ref{thm:Q=P}, the three parallelograms being $\mathcal{Q}_j$ for $j \in \{1,2,3\}$.  The inner vertex is $w = (y_2, y_3)$.

\begin{mycoro}
The polytope $\mathcal{P}$ is a zonotope (i.e., a Minkowski sum of line segments).
\end{mycoro}
This is obvious because the definition of $\mathcal{Q}$ via \eqref{eq:polytopedecomposition} is precisely one of the standard definitions of a zonotope.  See Lecture 7 of \cite{Ziegler} for more information on zonotopes.  An alternate definition of a zonotope is the image of a cube under an affine projection (see Definition 7.13 of \cite{Ziegler}); we give a concrete description of this in Remark \ref{remark:An} below.

\begin{mycoro}
\label{coro:polytopevolume}
The volume of the polytope $\mathcal{P}$ defined by \eqref{eq:system} is
\begin{equation}
\label{eq:volumepolynomial}
\sum_{j=1}^{n} y_1 \dots y_{j-1} y_{j+1} \dots y_n.
\end{equation}
\end{mycoro}
It is easy to express the volume of $\mathcal{Q}_j$ as the absolute value of the determinant of $(y_1 v_1^T, \dots, y_{j-1} v_{j-1}^T, y_{j+1}  v_{j+1}^T, \dots, y_n v_n^T)$ which is easily calculated to be $y_1 \dots y_{j-1} y_{j+1} \dots y_n$.

In general it is difficult to find the volume of a polytope defined parametrically as in \eqref{eq:system}.  There are formulas for the volume of a polytope if one knows the vertex set and its connecting edges but since our polytope is defined by half-planes this information is not given to us directly.
We also observe that the polynomial \eqref{eq:volumepolynomial} equals the Schur polynomial $S_{\lambda}(y_1, \dots, y_n)$ associated to the partition $\lambda = (1,1, \dots, 1, 0) \in \mathbb{Z}^n$.
  
\begin{myremark}
\label{remark:An}
 The polytope $\mathcal{P}$ has a relation to the $A_{n}$ lattice as we explain.  Recall the definition of the $A_{n}$ lattice as
\begin{equation}
 A_{n} = \{ (x_0, x_1, \dots, x_{n}) \in \mathbb{Z}^{n+1} : x_0 + \dots + x_{n} = 0 \},
\end{equation}
which defines an $n$-dimensional lattice in $\mathbb{R}^{n+1}$; see \cite{ConwaySloane}, Chapter 4.6.1.  It is the root lattice for $\mathfrak{sl}_{n+1}$.  It has generator matrix 
\begin{equation}
M= \begin{pmatrix}
   -1 & 1 & 0 & 0 & \dots & 0 & 0 \\
   0 & -1 & 1 & 0 & \dots & 0 & 0 \\
  0 & 0 & -1 & 1 & \dots & 0 & 0 \\
 \cdot & \cdot & \cdot & \cdot & \dots & \cdot & \cdot \\
0 & 0 & 0 & 0 & \dots & -1 & 1
 \end{pmatrix},
\end{equation}
where the $n$ rows, say $w_1, \dots, w_n$, respectively, are generators for the lattice.  A fundamental domain for $A_n$ inside the hyperplane $x_0 + \dots + x_n = 0$ is then $\{ t_1 w_1 + \dots + t_n w_n : 0 \leq t_j \leq 1 \text{ for all $j$} \}$.  Stretching the vector $w_j$ by $y_j$ for each $j$, we obtain the region $\mathcal{R}$ of the form 
\begin{equation}
\mathcal{R} := \{t_1 y_1 w_1 + \dots + t_n y_n w_n : 0 \leq t_j \leq 1 \text{ for all $j$} \}.
\end{equation}
Then $\mathcal{Q}-w$ is the projection of $\mathcal{R}$ via $x_0 = x_n = 0$, as can be seen since deleting the first and last columns of $M$ gives a matrix whose rows are $v_1, \dots, v_n$.  Note that this projection reduces the dimension of $\mathcal{R}$ by one.
\end{myremark}

\begin{proof}[Proof of Theorem \ref{thm:Q=P}]
First we show $\mathcal{Q} \subset \mathcal{P}$.
We can write \eqref{eq:polytopedecomposition} in the form
\begin{equation}
\label{eq:polytopematrixformula}
 \begin{pmatrix}
  x_1 \\ x_2 \\ \vdots \\ x_{n-1}
 \end{pmatrix}
= \begin{pmatrix}
   t_1 & 1- t_2 &  && \dots &  \\
  & t_2 & 1-t_3 &  & \dots &  \\
  &  & t_3 & 1-t_4 & \dots &  \\
 \vdots & & &\ddots  & & \\
  &  & & & t_{n-1} & 1-t_n
  \end{pmatrix}
\begin{pmatrix}
  y_1 \\ y_2 \\ \vdots \\ y_{n}
 \end{pmatrix}.
\end{equation}
By adding consecutive rows of the matrix, one sees that
\begin{equation}
\label{eq:xjtj}
 x_j + \dots + x_{k-1} = t_j y_j + (y_{j+1} + \dots + y_{k-1}) + (1-t_{k}) y_k,
\end{equation}
so that if $0 \leq t_1, \dots, t_n \leq 1$ then \eqref{eq:system} holds.  That is, $\mathcal{Q} \subset \mathcal{P}$.

Next we explain why $\cup_{j=1}^{n} \mathcal{Q}_j = \mathcal{Q}$ and $\mathcal{Q}_j \cap \mathcal{Q}_k$ is a subset of an $(n-2)$-dimensional cone in $\mathbb{R}^{n-1}$ for $j \neq k$.
Let $e_1,
\dots, e_{n-1}$ denote the standard basis vectors of $\mathbb{R}^{n-1}$.  Then \eqref{eq:polytopedecomposition} is the same as
\begin{equation}
\label{eq:cones}
x - w = 
u_1 w_1 + \dots + u_n w_n,
\end{equation}
where $u_j = t_j y_j$, $w_1 = e_1, w_2 = e_2-e_1, \dots, w_{n-1} = e_n - e_{n-1}, w_n = -e_n.$  Observe that $w_1 + \dots + w_{n} = 0$, and any set of $n-1$ $w_j$'s is a basis for $\mathbb{R}^{n-1}$.
For each $j \in \{1, \dots, n \}$, let
\begin{equation}
C_j = \{ u_1 w_1 + \dots + u_{j-1} w_{j-1} + u_{j+1} w_{j+1} + \dots + u_n w_n : u_k \geq 0 \text{ for all $k$} \}.
\end{equation}
Then $C_j$ is a cone in $\mathbb{R}^{n-1}$.  We claim that $\cup_{j=1}^{n} C_j = \mathbb{R}^{n-1}$ and $C_j \cap C_k$ is an $(n-2)$-dimensional cone for $j \neq k$.  To prove the claim, first suppose $v \in \mathbb{R}^{n-1}$, and express it (non-uniquely) as $v = u_1 w_1 + \dots + u_{n} w_{n}$ with $u_j \in \mathbb{R}$.  Then
\begin{equation}
\label{eq:vformula}
v = (u_1 + q) w_1 + \dots + (u_{n}+q) w_{n},
\end{equation}
for any $q \in \mathbb{R}$; choosing $q = \max(-u_1, -u_2, \dots, -u_n)$ gives $u_j + q \geq 0$ for all $j$, and $u_{j_0} + q = 0$ for some $j_0$, whence $v \in C_{j_0}$.  Next suppose $v \in C_j \cap C_k$ for some $j \neq k$.  By symmetry, suppose $j=n$, and $k=n-1$.  Then for some $u_j, u_j' \geq 0$, we have
\begin{align}
v = u_1 w_1 + \dots + u_{n-1} w_{n-1} &= u_1' w_1 + \dots + u_{n-2}' w_{n-2} + u_{n}' w_n,
\\
 &= (u_1' - u_n') w_1 + \dots + (u_{n-2}'-u_n') w_{n-2} - u_{n}' w_{n-1}.
\end{align}
Since $w_1, \dots w_{n-1}$ form a basis, we have in particular that $u_{n-1} = - u_n'$, but since $u_{n-1} \geq 0$ and $u_n' \geq 0$, we conclude $u_{n-1} = u_n' = 0$ and hence $v$ lies in the cone spanned by $w_1, \dots, w_{n-2}$.  Also, any element in this cone is an element of $C_n \cap C_{n-1}$.
The above claim immediately shows that $\mathcal{Q}_j \cap \mathcal{Q}_k$ lies in an $(n-2)$-dimensional cone for $j \neq k$, since $\mathcal{Q}_j \subset C_j$ for all $j$.  Also, the proof can be modified to show that $\cup_{j=1}^{n} \mathcal{Q}_j = \mathcal{Q}$: suppose $v \in \mathcal{Q}$ is given in the form \eqref{eq:vformula} with $0 \leq u_j \leq y_j$.  Then choosing $q = \max(-u_1, -u_2, \dots, -u_n)$ shows $v \in \mathcal{Q}_j$ for some $j$.


Finally, we need to show that $\mathcal{Q} = \mathcal{P}$.
It seems tricky to do this directly and our strategy is to show that every facet of $\mathcal{P}$ is contained in $\mathcal{Q}$; this then shows $\mathcal{P} \subset \mathcal{Q}$ because $\mathcal{P}$ is the convex hull of its facets (indeed, it is the convex hull of its vertices by the ``main theorem'' of polytopes: see Theorem 1.1 of \cite{Ziegler}).  More precisely, given $j_0, k_0 \in \{1, \dots, n \}$ with $j_0 \neq k_0$, consider the 
facet $\mathcal{F}_{j_0, k_0}$ of $\mathcal{Q}_{j_0}$ which equals the set of points of the form \eqref{eq:polytopedecomposition} with $t_{j_0} = 0$ and $t_{k_0} = 1$.  We will show that every facet of $\mathcal{P}$ equals $\mathcal{F}_{j_0, k_0}$ for some choice of $j_0, k_0$.  First we argue that $\mathcal{Q}_{j_0}$ is given by the following system of $n-1$ equations:
\begin{align}
\label{eq:xjtj0left}
 x_j + \dots + x_{j_0-1} &= t_j y_j + (y_{j+1} + \dots + y_{j_0}), \qquad 1 \leq j < j_0 \\
\label{eq:xjtj0right}
x_{j_0} + \dots + x_{k-1} &= y_{j_0+1} + \dots + y_{k-1} + (1-t_k) y_k, \qquad j_0 < k \leq n,
\end{align}
where $0 \leq t_j \leq 1$ for all $j \neq j_0$.  First we note that in general, \eqref{eq:xjtj} is equivalent to \eqref{eq:polytopedecomposition}.  Setting $t_{j_0} = 0$, and taking only \eqref{eq:xjtj} with $j=j_0$ or $k=j_0$, gives precisely the equations \eqref{eq:xjtj0left} and \eqref{eq:xjtj0right}.  This shows $\mathcal{Q}_{j_0}$ is contained in the set of solutions to \eqref{eq:xjtj0left} and \eqref{eq:xjtj0right}.  On the other hand, any solution to \eqref{eq:xjtj0left} and \eqref{eq:xjtj0right} uniquely determines $t_j$'s, $j \neq j_0$ with $0 \leq t_j \leq 1$, and therefore gives a solution to \eqref{eq:polytopematrixformula}.

Now suppose that $j_0 < k_0$, say.  If we impose the extra condition $t_{k_0} = 1$, then \eqref{eq:xjtj0left}-\eqref{eq:xjtj0right} show that $\mathcal{F}_{j_0, k_0}$ is given by the following system:
\begin{align}
\label{eq:xjtj0tk0left}
 x_j + \dots + x_{j_0-1} &= t_j y_j + (y_{j+1} + \dots + y_{j_0}), \qquad 1 \leq j < j_0 \\
\label{eq:xjtj0tk0middle}
x_{j_0} + \dots + x_{k-1} &= y_{j_0+1} + \dots + y_{k-1} + (1-t_k) y_k, \qquad j_0 < k < k_0, 
\\
\label{eq:xjtj0tk0face}
x_{j_0} + \dots + x_{k_0-1} &= y_{j_0+1} + \dots + y_{k_0-1}
\\
\label{eq:xjtj0tk0right}
x_{k_0} + \dots + x_{k-1} &= y_{k_0} + y_{k_0+1} + \dots + y_{k-1} + (1-t_k) y_k, \qquad k_0 < k \leq n.
\end{align}
Next we show that the facet of $\mathcal{P}$ defined by \eqref{eq:xjtj0tk0face} (which is indeed seen to be a facet by taking \eqref{eq:system} with $j=j_0, k=k_0$) is equivalent to the above system.  Suppose that $x$ satisfies \eqref{eq:system} and \eqref{eq:xjtj0tk0face}.  First we show \eqref{eq:xjtj0tk0left} with some $t_j \in [0,1]$ (the only issue is that $t_j$ lies in this interval).  It is immediate from \eqref{eq:system} that $x_j + \dots + x_{j_0-1} \leq y_{j} + \dots + y_{j_0}$ so $t_j \leq 1$.  To see $t_j \geq 0$, write $x_j + \dots + x_{j_0-1} = x_j + \dots + x_{k_0-1} - (x_{j_0} + \dots + x_{k_0-1})$ and apply \eqref{eq:system} and \eqref{eq:xjtj0tk0face} to these two terms, respectively, so that we see
\begin{equation}
 x_j + \dots + x_{j_0-1} \geq y_{j+1} + \dots + y_{k_0-1} - (y_{j_0+1} + \dots + y_{k_0-1}) = y_{j+1} + \dots + y_{j_0}.
\end{equation}
For \eqref{eq:xjtj0tk0middle}, we obtain $t_k \leq 1$ immediately from the lower bound in \eqref{eq:system} (with $j=j_0$).  The bound $t_k \geq 0$ follows by writing $x_{j_0} + \dots + x_{k-1} = x_{j_0} + \dots + x_{k_0-1} - (x_k + \dots + x_{k_0-1})$ and using \eqref{eq:xjtj0tk0face} and \eqref{eq:system}, respectively, to obtain
\begin{equation}
 x_{j_0} + \dots + x_{k-1} \leq y_{j_0+1} + \dots + y_{k_0-1} - (y_{k+1} + \dots + y_{k_0-1}) = y_{j_0+1} + \dots + y_{k}.
\end{equation}
The final case of \eqref{eq:xjtj0tk0right} is similar.  The condition $t_k \geq 0$ is immediate from \eqref{eq:system}.  The upper bound $t_k \leq 1$ uses $x_{k_0} + \dots + x_{k-1} = x_{j_0} + \dots + x_{k-1} - (x_{j_0} + \dots + x_{k_0-1})$ and \eqref{eq:system} and \eqref{eq:xjtj0tk0face}, respectively, to obtain
\begin{equation}
 x_{k_0} + \dots + x_{k-1} \geq y_{j_0+1} + \dots + y_{k-1} - (y_{j_0+1} + \dots + y_{k_0-1}) = y_{k_0} + \dots + y_{k-1}.
\end{equation}
We have thus shown that any point $x \in \mathcal{P}$ that lies on the facet \eqref{eq:xjtj0tk0face} lies in $\mathcal{F}_{j_0, k_0}$.  So far we have not treated the opposite facet $x_{j_0} + \dots x_{k_0-1} = y_{j_0} + \dots + y_{k_0}$, but a symmetry argument suffices here as we now explain.  One can check that the change of variables $x_j = y_j + y_{j+1} - x_j'$, applied to the system \eqref{eq:system}, leads to the same system \eqref{eq:system} (in terms of the new variables $x_j'$) but with each of the opposite facets reversed.  This change of variables, combined with $t_j = 1 - t_j'$, also leaves \eqref{eq:polytopedecomposition} invariant.  Thus the opposite facets of $\mathcal{P}$ also occur as facets of $\mathcal{Q}_{j_0}$.
\end{proof}

\section{The lower bound}
\label{section:lowerbound}
\subsection{}
\label{subsection:lowerbound}
The lower bound for $N_d(F)$ is, in a combinatorial sense, much more difficult than the upper bound (ignoring the major assumption of Lindel\"{o}f which is required for the upper bound).  The difference is that we need much more refined information about the spectral sum in \eqref{eq:normUBspectraldecomposition}.  This was largely obtained in Section \ref{section:polytope}.

\begin{myprop}
\label{prop:normLBwithoutLfunction}
 Assume \eqref{eq:WLWLlower} holds and $|\alpha_j - \alpha_k| \geq \lambda(F)^{\varepsilon}$ for all $j \neq k$.  Then
\begin{equation}
 \sumstar_j |B_j(1)|^2 \intR q(t, \alpha, \beta_j) dt \gg 1.
\end{equation}
\end{myprop}
The left hand side above is the unweighted version of \eqref{eq:normUBspectraldecomposition}, up to the factor $|A_F(1)|^2$.

\begin{mylemma}
\label{lemma:qtLB}
 Suppose that $\alpha = (\alpha_1, \dots, \alpha_{n+1}) \in \mr^{n+1}$, $\beta = (\beta_1, \dots, \beta_n) \in \mr^n$ and $\alpha_1 \geq \alpha_2 \geq \dots \geq \alpha_{n+1}$, $\beta_1 \geq \beta_2 \geq \dots \geq \beta_n$, and that \eqref{eq:alphajbetakcondition} holds.  Also suppose that $\sup S_- =: \smax $ and $\inf S_{+} =: \smin$.  
Then for $t \in I_M$, we have
\begin{equation}
\label{eq:qtLB}
 q(t,\alpha, \beta) \geq \prod_{s_+ \in S_+} (1+ |s_{+} - \smax |)^{-1/2} \prod_{s_- \in S_-+} (1+ |s_{-} - \smin |)^{-1/2}.
\end{equation}
Furthermore,
\begin{equation}
\label{eq:qtintegralLB}
 \intR q(t, \alpha, \beta) dt \gg |I_M| \prod_{s_+ \in S_+} (1+ |s_{+} - \smax |)^{-1/2} \prod_{s_- \in S_-+} (1+ |s_{-} - \smin |)^{-1/2}.
\end{equation}
\end{mylemma}
\begin{proof}
 We note
\begin{equation}
 1 + |t+ s_+| = 1 + t + s_+ = 1 + t + \smax + (s_+ - \smax) \leq 1  + (s_+ - \smax).
\end{equation}
A similar bound holds for $s_-$, so \eqref{eq:qtLB} is shown; \eqref{eq:qtintegralLB} follows immediately. 
\end{proof}

\begin{proof}[Proof of Proposition \ref{prop:normLBwithoutLfunction}]
We need to show that
\begin{equation}
\label{eq:lowerestimatewithbetaj}
 \sumstar_{\beta_j \text{ admissible}} |B_j(1)|^2 |I_M| \prod_{s_+ \in S_+} (1+ |s_{+} - \smax |)^{-1/2} \prod_{s_- \in S_-+} (1+ |s_{-} - \smin |)^{-1/2} \gg 1,
\end{equation}
where $\beta_j$ admissible means that \eqref{eq:alphajbetakcondition} holds.

The work in Section \ref{section:polytope} is precisely what we need to obtain control over the various parameters above.  Recall that we used the notation $x_j = \beta_j - \beta_{j+1}$ and $y_j = \alpha_{n+1-j} - \alpha_{n+2-j}$.  With this notation, $\beta_j$ being admissible is equivalent to $x \in \mathcal{P}$.  Theorem \ref{thm:Q=P} provides a useful decomposition of $\mathcal{P}$ into parallelohedra; in particular, we shall restrict attention to $\mathcal{Q}_{j_0} \subset \mathcal{P}$ where $j_0 \in \{1, \dots, n \}$ is chosen to minimize $y_{j_0}$ (amongst all other choices of $y_j$).  We showed that $\mathcal{Q}_{j_0}$ is parameterized by \eqref{eq:xjtj0left} and \eqref{eq:xjtj0right}, where $0 \leq t_j \leq 1$ for all $j$.  Define $\mathcal{Q}_{j_0}^* \subset \mathcal{Q}_{j_0}$ to be given by \eqref{eq:xjtj0left} and \eqref{eq:xjtj0right} but further restricted by $\frac14 \leq t_j \leq \frac34$ for all $j$.  For $x \in \mathcal{Q}_{j_0}^*$ we can simplify \eqref{eq:lowerestimatewithbetaj}.

In general for $x \in \mathcal{P}$, we have
\begin{equation}
 |I_M| = \min\{ \alpha_{n+1-j} + \beta_j - \alpha_{n+2-k} - \beta_k : 1 \leq j, k \leq n \}.
\end{equation}
If $j \leq k$ then 
\begin{equation}
\alpha_{n+1-j} + \beta_j - \alpha_{n+2-k} - \beta_k = x_j + \dots + x_{k-1} - (y_{j+1} + \dots + y_{k-1}) = t_j y_j + (1-t_k)y_k,
\end{equation}
using \eqref{eq:xjtj}.
If $j \geq k$ then switching $j$ and $k$ gives 
\begin{equation}
\alpha_{n+1-k} + \beta_k - \alpha_{n+2-j} - \beta_j = 
y_j + \dots + y_k - (x_j + \dots + x_{k-1}) = (1-t_j) y_j + t_k y_k.
\end{equation}
Since $t_{j_0} = 0$ and $y_{j_0}$ is minimal, 
we conclude that for $x \in \mathcal{Q}_{j_0}^*$, $|I_M| \geq \frac14  y_{k_0}$, say where $k_0 \neq j_0$ is the second smallest of the $y_j$'s, after $y_{j_0}$.
Actually, in general, we have for $x \in \mathcal{Q}_{j_0}^*$ that
\begin{equation}
\frac18(y_j + y_k) \leq \alpha_{n+1-j} + \beta_j - \alpha_{n+2-k} - \beta_k \leq \frac{3}{4}(y_j + y_k).
\end{equation}

Next we estimate the terms in \eqref{eq:lowerestimatewithbetaj} with $s_+$ and $s_-$ not lying immediately above or below a horizontal line in \eqref{eq:alphabetaarray}.  We will show
\begin{equation}
\label{eq:alphakbetalawayfromdiagonal}
\prod_{l+m \leq n} (1 + |\alpha_l + \beta_m - \smax|)^{-1/2} \prod_{l+m \geq n+3} (1 + |\alpha_l + \beta_m - \smin|)^{-1/2} \gg \frac{1}{\mu(\beta)}.
\end{equation}
Suppose $l+m \leq n$.  Then we estimate the term with $s_+ = \alpha_l + \beta_m$ in \eqref{eq:lowerestimatewithbetaj} by comparing with $\alpha_l + \beta_{n+1-l}$ which is the entry in \eqref{eq:alphabetaarray} in the same row as $s_+$ and directly above one of the horizontal lines.  We note $\alpha_l + \beta_{n+1-l} - \smax \leq \alpha_l + \beta_{n+1-l} - (\alpha_{l+1} + \beta_{n+1-l})$, since $\alpha_l + \beta_{n+1-l}$ lies below the horizontal lines in \eqref{eq:alphabetaarray} while $\smax$ is the largest entry below the horizontal lines.  Taken together, we deduce
\begin{multline}
1 + \alpha_l + \beta_m - \smax = 1 + (\beta_m - \beta_{n+1-l}) + (\alpha_l + \beta_{n+1-l} - \smax).
\\
\leq 1 + (\beta_m - \beta_{n+1-l}) + (\alpha_l - \alpha_{l+1}).
\end{multline}
Now we note that by \eqref{eq:xjtj},
\begin{equation}
\label{eq:betambetan}
\beta_m - \beta_{n+1-l} = x_m + \dots + x_{n-l} \geq \frac14  y_{n+1-l} = \frac14 (\alpha_l - \alpha_{l+1}).  
\end{equation}
Hence, for $l + m \leq n$ we have
\begin{equation}
\label{eq:smaxupperbound}
1 + \alpha_l + \beta_m - \smax \leq 5 (1 + (\beta_m - \beta_{n+1-l})).
\end{equation}
Similarly, if $l +m \geq n+3$ we have
\begin{multline}
 1 + \smin - \alpha_l - \beta_m = 1 + (\beta_{n+2-l} - \beta_m) + (\smin - \alpha_l - \beta_{n+2-l}) 
\\
\leq 1 + (\beta_{n+2-l} - \beta_m) + (\alpha_{l-1} - \alpha_l).
\end{multline}
By \eqref{eq:xjtj}, we have
\begin{multline}
 \beta_{n+2-l} - \beta_m = t_{n+2-l} y_{n+2-l} + y_{n+3-l} + \dots + y_{m-1} + (1-t_m) y_m \geq t_{n+2-l} y_{n+2-l} + \frac14 y_m.
\end{multline}
This case is slightly different from \eqref{eq:betambetan} because if $n+2-l = j_0$ then $t_{n+2-l} = 0$.  However, if $n+2-l = j_0$ then $y_m \geq y_{j_0}$ by minimality, so we can always conclude $\beta_{n+2-l} - \beta_m \geq \frac14 (\alpha_{l-1} - \alpha_{l})$, so we have
\begin{equation}
\label{eq:sminupperbound}
1 + \smin - \alpha_l - \beta_m \leq 5 ( 1 + \beta_{n+2-l} - \beta_m).
\end{equation}
Then putting together \eqref{eq:smaxupperbound} and \eqref{eq:sminupperbound}, we deduce \eqref{eq:alphakbetalawayfromdiagonal} by noting that a term $|\beta_k - \beta_l |$ occurs exactly once in each of \eqref{eq:smaxupperbound} and \eqref{eq:sminupperbound}, for $k \neq l$.  

Now we turn to the terms with $k+l = n+1$ and $k+l = n+2$.  If $k \neq j_0, k_0$ then we use
\begin{equation}
1 + \alpha_{n+1-k} + \beta_{k} - \smax \leq 1 + \alpha_{n+1-k} + \beta_{k} - (\alpha_{n+2-k} + \beta_{k}) = 1 + y_k.
\end{equation}
For $k = j_0$ we use
\begin{equation}
1 + \alpha_{n+1-j_0} + \beta_{j_0} - \smax \leq 1 + \alpha_{n+1-j_0} + \beta_{j_0} - \alpha_{n+2-k_0} - \beta_{k_0} \asymp y_{j_0} + y_{k_0}.
\end{equation}
Similarly for $k = k_0$.  Then we obtain
\begin{multline}
\label{eq:alphakbetalalongdiagonal}
\prod_{k+l=n+1} (1 + |\alpha_k + \beta_l - \smax|)^{-1/2} \prod_{k+l = n+2} (1 + |\alpha_k + \beta_l - \smin|)^{-1/2} 
\\
\gg (y_1 \dots y_{j_0-1} (y_{j_0} + y_{k_0}) y_{j_0 + 1} \dots y_{k_0-1} (y_{j_0} + y_{k_0}) y_{k_0+1} \dots y_n)^{-1}.
\end{multline}

By combining \eqref{eq:alphakbetalawayfromdiagonal} with \eqref{eq:alphakbetalalongdiagonal} and our lower bound $|I_M| \gg y_{j_0} + y_{k_0}$, we have that the left hand side of \eqref{eq:lowerestimatewithbetaj} is
\begin{equation}
\label{eq:lowerestimatesimplified}
 \gg \frac{y_{j_0} + y_{k_0}}{y_1 \dots y_{j_0-1} y_{j_0+1} \dots y_{k_0-1} (y_{j_0} + y_{k_0})^2 y_{k_0+1} \dots y_n} \sum_{\beta_j : x \in \mathcal{Q}_{j_0}^*} |B_j(1)|^2 \frac{1}{\mu(\beta_j)}.
\end{equation}
The weighted local Weyl law \eqref{eq:WLWLlower} counts the number of $\beta_j$ in a small box, weighted by $|B_j(1)|^2$, so calculating the sum amounts to finding the number of integer vectors $x$ that lie in $\mathcal{Q}_{j_0}^*$ (which is essentially the volume of $\mathcal{Q}_{j_0}^*$ since it is a parallelohedron with sides longer than $\lambda(F)^{\varepsilon}$).  The volume of $\mathcal{Q}_{j_0}^*$ is
 $\gg y_1 \dots y_{j_0-1} y_{j_0+1} \dots y_{k_0-1} (y_{j_0} + y_{k_0}) y_{k_0+1} \dots y_n$, and hence \eqref{eq:lowerestimatesimplified} is
\begin{equation}
\gg (y_{j_0} + y_{k_0}) \frac{ y_1 \dots y_{j_0-1} y_{j_0+1} \dots y_{k_0-1} (y_{j_0} + y_{k_0}) y_{k_0+1} \dots y_n}{y_1 \dots y_{j_0-1} y_{j_0+1} \dots y_{k_0-1} (y_{j_0} + y_{k_0})^2 y_{k_0+1} \dots y_n}  \asymp 1. \qedhere
\end{equation}
\end{proof}

Finally, we need to deduce Theorem \ref{thm:normLB} from Proposition \ref{prop:normLBwithoutLfunction}.
It is not difficult to lower bound the second moment of $L$-functions as we codify with the following:
\begin{myprop}
\label{prop:secondmomentLowerBound}
Let $T_0 \in \mathbb{R}$, $T > 0$, let $L(f, s)$ be an $L$-function of degree $d$.  Suppose that for $s=1/2 + it$ with $T_0-T \leq t \leq T_0 + T$, the analytic conductor $q(f,s)$ of $L(f,s)$ satisfies $q(f,s) \leq Q$.  Then with $T \geq Q^{\varepsilon}$, we have
\begin{equation}
 \int_{T_0-T}^{T_0+T} |L(f, 1/2 + it)|^2 dt \geq T + O_{d,\varepsilon}(Q^{-99}).
\end{equation}
\end{myprop}

With Proposition \ref{prop:secondmomentLowerBound}, it only takes a very small modification to prove Theorem \ref{thm:normLB}.  The basic idea is that for $t \in I_M$, \eqref{eq:qtLB} provides a lower bound on $q(t,\alpha, \beta)$, while for $x \in \mathcal{Q}_{j_0}^*$ then as we showed in the proof of Proposition \ref{prop:normLBwithoutLfunction}, we have $|I_M| \gg (\alpha_{k_0} - \alpha_{k_0-1})$, which is $\gg \lambda(F)^{\varepsilon}$ by the spacing condition assumed in Theorem \ref{thm:normLB}.  Thus
\begin{equation}
 \int_{t \in I_M} q(t, \alpha, \beta) |L(1/2 + it, F \times u_j)|^2 dt \gg |I_M| \prod_{s_+ \in S_+} (1+ |s_{+} - \smax |)^{-1/2} \prod_{s_- \in S_-+} (1+ |s_{-} - \smin |)^{-1/2},
\end{equation}
and since in our proof of Proposition \ref{prop:normLBwithoutLfunction} we showed \eqref{eq:lowerestimatewithbetaj}, we finish the deduction.

\begin{proof}[Proof of Proposition \ref{prop:secondmomentLowerBound}]
 By a standard one-piece approximate functional equation, we have
\begin{equation}
 L(f, 1/2 + it) = \sum_{n \leq Q^{1+\varepsilon}} \frac{\lambda_f(n)}{n^{1/2 + it}} V(n) + O(Q^{-100}),
\end{equation}
where $V$ depends on the degree $d$ but not on $f$.  One can choose $V$ so that $V(1) = 1 + O(Q^{-100})$.  Suppose $w$ is a fixed smooth, nonnegative function such that $w(x) = 1$ for $|x| <  1/2$ and $w(x) = 0$ for $|x| > 1$.  Then by Cauchy's inequality and positivity,
\begin{equation}
 \int_{T_0-T}^{T_0 + T} |L(f, 1/2 + it)|^2 dt \geq \frac{\big|\intR L(f, 1/2 + it) w(\frac{t-T_0}{T}) dt \big|^2}{\intR w(\frac{t-T_0}{T}) dt}.
\end{equation}
The denominator above is $T \widehat{w}(0)$.  By the approximate functional equation, the numerator is
\begin{equation}
 \sum_{n \leq Q^{1+\varepsilon}} \frac{\lambda_f(n)}{n^{1/2}} V(n) \intR n^{-it} w(\frac{t-T_0}{T}) dt + O(Q^{-99}).
\end{equation}
The inner integral simplifies as $T n^{-iT_0} \widehat{w}(\frac{T \log{n}}{2\pi})$, which for $n > 1$ is $\ll T^{-A}$ for $A$ arbitrarily large.  Taking $A$ large compared to $\varepsilon$, we can trivially bound the terms $n > 1$ by $O(Q^{-100})$.  Thus
\begin{equation}
 \int_{T_0-T}^{T_0+T} |L(f, 1/2 + it)|^2 dt \geq T \widehat{w}(0) + O(Q^{-99}).
\end{equation}
Since $\widehat{w}(0) \geq \int_{-1/2}^{1/2}  dx = 1$, we complete the proof.
\end{proof}

\subsection{Approaching the walls}
\label{section:walls}
It seems interesting to understand the behavior of $N(F)$ as the spectral parameters of $F$ may become close; that is, dropping the assumption that $|\alpha_k - \alpha_l| \geq \lambda(F)^{\varepsilon}$ for all $k \neq l$.  The upper bound did not need this assumption so we only consider the lower bound.  Recall \eqref{eq:qtalphabetaSTIRLING}.  Instead of restricting attention to the set where $r(t,\alpha,\beta) = 0$ (which may have very small measure in this degenerate situation), we consider the set where $r(t,\alpha, \beta) \geq - C'$ where $C' > 1$ may depend on $n$ only.  If $I_M = [a,b]$ then define $I_M^C = [a-C, b+C]$; then for $t \in I_M^C$, we have $r(t,\alpha,\beta) \geq -n(n+1)C$.  We also relax the condition \eqref{eq:alphajbetakcondition} to (say) 
\begin{equation}
 \label{eq:extendedalphajbetakcondition}
\alpha_{n+1-k} - \alpha_{n+2-j} + C \geq \beta_j - \beta_k \geq \alpha_{n+2-k} - \alpha_{n+1-j} - C.
\end{equation}
If $t \in I_M^C$ and \eqref{eq:extendedalphajbetakcondition} holds, then $r(t,\alpha,\beta) \geq -C'$ for some $C' > 1$ depending only on $n$ and the choice of $C > 0$.  If we define $y_j = \max(\alpha_{n+1-j} - \alpha_{n+2-j}, C'')$ for some $C'' > 1$ then \eqref{eq:system} implies that \eqref{eq:extendedalphajbetakcondition} holds (for some fixed $C > 0$).  At this point the work in Section \ref{subsection:lowerbound} carries through almost unchanged.

\section{Asymptotics}
\label{section:asymptotic}
In this section we let $c$ denote a positive constant that may change line to line (to avoid excessive re-labelling).  We have
\begin{equation}
 N(F) \sim c |A_F(1)|^2 \sum_{j} \intR |L(1/2 + it, F \times \overline{u_j})|^2 q(t, \alpha, \beta_j) dt + \dots
\end{equation}
with the dots indicating $N_{\text{max}}(F) + N_{\text{min}}(F)$.  Recall that the $L$-function has Dirichlet series given by \eqref{eq:RankinSelbergCuspidal}, and $q$ is defined by \eqref{eq:qtalphabetaDEF}.  In fact, $q$ contains the appropriate gamma factors for the completed $L$-function.

We wish to apply the moment conjectures of \cite{CFKRS} to find the asymptotic of $N(F)$.  For this, we shift slightly and consider
\begin{multline}
 \sumstar_j \intR q(t-iz,\alpha, \beta_j) \sum_{m_1, m_1' \geq 1} \dots \sum_{m_{n}, m_n' \geq 1}  
\\
\frac{\lambda_F(m_1, \dots, m_n) \overline{\lambda_F}(m_1', \dots, m_n') \lambda_j(m_2, \dots, m_n) \overline{\lambda_j}(m_2', \dots, m_n')}{\prod_{k=1}^{n} m_k ^{(n+1-k)(\half + it+ z)} (m_k')^{(n+1-k)(\half -it+ z)}}
+ \dots,
\end{multline}
the dots indicating an identical term under $z \rightarrow -z$.  Together, the $j$-sum and the $t$-integral should pick out the diagonal terms with $m_k = m_k'$ for all $1 \leq k \leq n$.  Thus we should have
\begin{equation}
 N(F) \sim c |A_F(1)|^2 \sum_{m_1, \dots, m_n} \frac{|\lambda_F(m_1, \dots, m_n)|^2}{m_1^{n(1+2z)} m_2^{(n-1)(1+2z)} \dots m_n^{1+2z}} \sumstar_j \intR q(t-iz, \alpha, \beta_j) dt + \dots.
\end{equation}
Note that this is
\begin{equation}
 N(F) \sim c |A_F(1)|^2 L(1 + 2z, F \times \overline{F}) \sumstar_j \intR q(t-iz,\alpha,\beta_j) dt + \dots.
\end{equation}
Note for $z$ small that
\begin{equation}
 \frac{q(t-iz, \alpha, \beta_j)}{q(t, \alpha, \beta_j)} = 1 -iz \frac{q'}{q}(t, \alpha, \beta_j) + O(z^2),
\end{equation}
and
\begin{equation}
 \frac{q'}{q}(t,\alpha,\beta_j) = \sum_{l=1}^{n} \sum_{k=1}^{n+1} \log |\frac{1/2 + it+ i\alpha_k + i\beta_{l}}{2} |^2 + O(1).
\end{equation}
Our work in Sections \ref{section:Archimedean} and \ref{section:polytope} describes the effective region of integration appearing in \eqref{eq:NFasymptotic}.
In particular, with some work, one can derive that $\frac{q'}{q}(t, \alpha, \beta_j) = \log \lambda(F) + O(1)$ inside the region of interest; the reason for this is that when $t \in I_M$ and $\beta$ lies in the polytope $\mathcal{P}$, the largest element $t + \alpha_1 + \beta_1$ is $\geq \alpha_1 - \alpha_n$. Typically, $\alpha_n \leq 0$ in which case this term is $\geq \alpha_1$.  If not, then we look at the smallest term
$|t+ \alpha_{n+1} + \beta_n| \geq \alpha_2 - \alpha_{n+1}$ which is $ - \alpha_{n+1}$ assuming $\alpha_2 \geq 0$.  Thus
$\frac{q'}{q} \geq \log \alpha_1^2 + \log \alpha_{n+1}^2 +  O(1) = \log \lambda(F) + O(1)$.  On the other hand, the upper bound $\frac{q'}{q} \leq \log \lambda(F) + O(1)$ is relatively trivial. 
Write $L(1 + 2z, F \times \overline{F}) = \frac{r}{2z} (1 + r'z + O(z^2))$.  
By Theorem 5.17 of \cite{IK} (conditional on GRH and Ramanujan), we have
 $r' \ll \log \log \lambda(F)$.  Since $q'/q \sim \log \lambda(F)$, we then derive the following asymptotic after letting $z \rightarrow 0$:
\begin{equation}
\label{eq:NFasymptotic}
 N(F) \sim c |A_F(1)|^2 r \sumstar_j \intR q'(t, \alpha, \beta_j) dt \sim c \sum_j \intR q'(t, \alpha, \beta_j) dt, 
\end{equation}
in view of Proposition \ref{prop:RankinSelbergL2formula}.   Thus,
\begin{equation}
 N(F) \sim c \log \lambda(F) \sumstar_j \intR q(t, \alpha, \beta_j) dt.
\end{equation}

Now it should be the case that $\sumstar_j q(t, \alpha, \beta_j) \sim c \int q(t, \alpha, \beta) \mu(\beta) d\beta$, where $\mu(\beta)$ is the spectral measure.
So we need to calculate the following integral
\begin{equation}
 \intR \int \mu(\beta) \frac{\prod_{l=1}^{n} \prod_{k=1}^{n+1} |\Gamma(\frac{1/2 + it+ i\alpha_k + i\beta_{l}}{2}) |^2}{\Big(\prod_{1 \leq k < l \le n+1} |\Gamma(\frac{1+ i\alpha_k - i\alpha_l}{2}) |^2\Big) \Big(\prod_{1 \leq k < l \le n} | \Gamma(\frac{1+ i\beta_{k} - i\beta_{l}}{2})|^2 \Big)} d\beta dt.
\end{equation}
Recall that $\beta_1 + \dots + \beta_n = 0$ and the $\beta$-integral is over this hyperplane.  Changing variables $\beta_k \rightarrow \beta_k - t$ for all $k$ reduces to calculating
\begin{equation}
 \int_{\mathbb{R}^{n}} \mu(\beta) \frac{\prod_{l=1}^{n} \prod_{k=1}^{n+1} |\Gamma(\frac{1/2 + i\alpha_k + i\beta_{l}}{2}) |^2}{\Big(\prod_{1 \leq k < l \le n+1} |\Gamma(\frac{1+ i\alpha_k - i\alpha_l}{2}) |^2\Big) \Big(\prod_{1 \leq k < l \le n} | \Gamma(\frac{1+ i\beta_{k} - i\beta_{l}}{2})|^2 \Big)} d\beta_1 \dots d \beta_n.
\end{equation}
The spectral measure is given by
\begin{equation}
 \mu(\beta) = c \prod_{1 \leq k < l \leq n} \frac{| \Gamma(\frac{1+ i\beta_{k} - i\beta_{l}}{2})|^2}{| \Gamma(\frac{i\beta_{k} - i\beta_{l}}{2})|^2},
\end{equation}
so that we need to calculate
\begin{equation}
\label{eq:Calpha}
c(\alpha) \int_{\mathbb{R}^{n}} \frac{\prod_{l=1}^{n} \prod_{k=1}^{n+1} |\Gamma(\frac{1/2 + i\alpha_k + i\beta_{l}}{2}) |^2}{ \prod_{1 \leq k < l \leq n} | \Gamma(\frac{i\beta_{k} - i\beta_{l}}{2})|^2 } d\beta_1 \dots d \beta_n, \quad c(\alpha) = \prod_{1 \leq k < l \le n+1} |\Gamma(\frac{1+ i\alpha_k - i\alpha_l}{2}) |^{-2}.
\end{equation}

We continue as follows.  By Stirling, the integral is
\begin{equation}
\label{eq:approximateintegralafterStirling}
\approx c \int_R \prod_{1 \leq k < l \leq n} (1+|\beta_k - \beta_l|) \prod_{l=1}^{n} \prod_{k=1}^{n+1} (1+|\alpha_k + \beta_l|)^{-1/2} d \beta,
\end{equation}
where the region of integration $R$ is defined by $\beta_j + \alpha_{n+1-j} \geq 0$ and $\beta_j + \alpha_{n+2-j} \leq 0$ for all $j$ (this is equivalent to the definition of $I_M$ after changing variables).  The restriction to $R$ should not alter the final answer very much because of the exponential decay of the gamma factors outside of $R$.
As noted with \eqref{eq:interlacing}, the region defined by $R$ is equivalent to
\begin{equation}
 -\alpha_{n+1} \geq \beta_1 \geq -\alpha_n \geq \beta_2 \geq \dots \geq -\alpha_{2} \geq \beta_n \geq -\alpha_1,
\end{equation}
which defines a box in $\mathbb{R}^n$.
By Lemma \ref{lemma:qtalphabetaUB}, \eqref{eq:approximateintegralafterStirling} is
\begin{multline}
 \ll \int_R \prod_{k=1}^n (1 + |\beta_k + \alpha_{n+1-k}|)^{-1/2} (1 + |\beta_k + \alpha_{n+2-k}|)^{-1/2} d \beta
\\
= \prod_{k=1}^n \int_{-\alpha_{n+1-k}}^{\alpha_{n+2-k}} (1 + |\beta_k + \alpha_{n+1-k}|)^{-1/2} (1 + |\beta_k + \alpha_{n+2-k}|)^{-1/2} d\beta_k.
\end{multline}
By \eqref{eq:hahb}, we then have that \eqref{eq:approximateintegralafterStirling} is $\ll 1$.
Thus we are led to the conjecture $N(F) \ll \log \lambda(F)$.  The lower bound of the same order of magnitude is contained in Section \ref{section:lowerbound}. 
One may observe that the above upper bound is a somewhat conceptually simpler version of the arguments given to prove Lemma \ref{lemma:qUB}.


\begin{thebibliography}{99}
\bibitem[B]{Blomer} V. Blomer, {\it Applications of the Kuznetsov formula on $GL(3)$}, preprint (2012) {\tt http://arXiv:1205.1781}.
\bibitem[BGT]{BGT} N. Burq, P. G\'{e}rard, N. Tzvetkov, {\it Restrictions of the Laplace-Beltrami eigenfunctions to submanifolds}, Duke Math. J. 138 (3) (2007) 445--486.
\bibitem[CP]{CP} J. Cogdell and I. Piatetski-Shapiro, {\it Remarks on Rankin-Selberg convolutions.}  Contributions to automorphic forms, geometry, and number theory,  255--278, Johns Hopkins Univ. Press, Baltimore, MD, 2004.
\bibitem[CFKRS]{CFKRS}  J.B. Conrey, D. Farmer, J. Keating, M. Rubinstein, and N. Snaith, {\it Integral moments of $L$-functions},  Proc. London Math. Soc. (3)  91  (2005),  no. 1, 33--104. 
\bibitem[CS]{ConwaySloane} J. H. Conway and N.J.A. Sloane, {\it Sphere Packings, Lattices, and Groups} Third edition. With additional contributions by E. Bannai, R. E. Borcherds, J. Leech, S. P. Norton, A. M. Odlyzko, R. A. Parker, L. Queen and B. B. Venkov. Grundlehren der Mathematischen Wissenschaften [Fundamental Principles of Mathematical Sciences], 290. Springer-Verlag, New York, 1999.
\bibitem[GRS]{GRS} A. Ghosh, A. Reznikov, and P. Sarnak, {\it Nodal domains of Maass forms I } preprint (2012) {\tt http://arxiv.org/abs/1207.6625}
\bibitem[Go]{Goldfeld} D. Goldfeld, {\it Automorphic forms and $L$-functions for the group ${\rm GL}(n,\mathbb{R})$.} With an appendix by Kevin A. Broughan. Cambridge Studies in Advanced Mathematics, 99. Cambridge University Press, Cambridge, 2006.
\bibitem[GW]{GoodmanWallach} R. Goodman and N. Wallach, {\it Symmetry, Representations, and Invariants} Symmetry, representations, and invariants. Graduate Texts in Mathematics, 255. Springer, Dordrecht, 2009.
\bibitem[GR]{GR} Gradshteyn, I. S.; Ryzhik, I. M. {\it Table of Integrals, Series, and Products}. Translated from the Russian. Sixth edition. Translation edited and with a preface by Alan Jeffrey and Daniel Zwillinger. Academic Press, Inc., San Diego, CA, 2000.
\bibitem[HL]{HL} J. Hoffstein and P. Lockhart, {\it Coefficients of Maass forms and the Siegel zero.} With an appendix by Dorian Goldfeld, Hoffstein and Daniel Lieman. Ann. of Math. (2) 140 (1994), no. 1, 161--181. 
\bibitem[Iw2]{IwaniecSpectral} H. Iwaniec, {\it Spectral methods of automorphic forms.} Second edition. Graduate Studies in Mathematics, 53. American Mathematical Society, Providence, RI; Revista Matemática Iberoamericana, Madrid, 2002. 
\bibitem[IK]{IK} H. Iwaniec and E. Kowalski, {\it Analytic Number Theory}. American Mathematical Society Colloquium Publications, 53.  American Mathematical Society, Providence, RI, 2004.
\bibitem[L]{Langlands} R. Langlands, {\it Problems in the theory of automorphic forms.} Lectures in modern analysis and applications, III, pp. 18–61. Lecture Notes in Math., Vol. 170, Springer, Berlin, 1970.
\bibitem[LM]{LapidMuller} E. Lapid and W. M\"{u}ller, {\it Spectral asymptotics for arithmetic quotients of {${\rm
              SL}(n,\mathbb{R})/{\rm SO}(n)$}}. Duke Math. J. 149 (2009), no. 1, 117--155. 
\bibitem[Li]{XiannanLi} Xiannan Li, {\it Upper bounds on L-functions at the edge of the critical strip.} Int. Math. Res. Not. IMRN 2010, no. 4, 727--755.
\bibitem[LY]{LiYoung} Xiaoqing Li and M. Young, {\it The $L^2$ restriction norm of a $GL_3$ Maass form}. Compos. Math., 148 , pp 675-717.
\bibitem[Ma]{Marshall} S. Marshall, {\it $L^p$ bounds for higher rank eigenfunctions and asymptotics of spherical functions}.  Preprint, 2011. {\tt  http://arxiv.org/abs/1106.0534}
\bibitem[Mi]{Milicevic} D. Mili\'{c}evi\'{c}, {\it Large values of eigenfunctions on arithmetic hyperbolic surfaces},  Duke Math. J. 155 (2010), no. 2, 365--401.
\bibitem[MW]{MW} C. M{\oe}glin and J.-L. Waldspurger, {\it Spectral decomposition and Eisenstein series.} Cambridge Tracts in Mathematics, 113. Cambridge University Press, Cambridge, 1995. 
\bibitem[Sa1]{SarnakMorowetz} P. Sarnak, {\it Letter to Morawetz}, 2004 {\tt http://www.math.princeton.edu/sarnak/}
\bibitem[Sa2]{SarnakReznikov} P. Sarnak, {\it Letter to Reznikov}, 2008 {\tt http://www.math.princeton.edu/sarnak/}
\bibitem[St1]{StadeAJM} E. Stade, {\it Mellin transforms of {${\rm GL}(n,\mathbb{R})$} {W}hittaker functions.} Amer. J. Math. 123 (2001), no. 1, 121--161.
\bibitem[St2]{StadeIsrael} E. Stade, {\it Archimedean L-factors on ${\rm GL}(n)\times {\rm GL}(n)$ and generalized Barnes integrals.}  Israel J. Math. 127 (2002), 201--219.
\bibitem[T]{Templier} N. Templier, {\it Large values of modular forms}, 
Preprint (2012) {\tt http://http://arxiv.org/abs/1207.6134}
\bibitem[W]{Waldspurger} J.-L. Waldspurger, {\it Sur les coefficients de Fourier des formes modulaires de poids demi-entier.} J. Math. Pures Appl. (9) 60 (1981), no. 4, 375--484. 
\bibitem[Ze]{Zelobenko} D. P. {\v{Z}}elobenko, {\it Compact Lie Groups and Their Representations}, Izdat. ``Nauka,'' Moscow, 1970 (Russian).  English translation: Translations of Mathematical Monographs, Vol. 40. American Mathematical Society, Providence, R.I., 1973.
\bibitem[Zh]{Zhang} S. W. Zhang, {\em Gross--Zagier formula for $GL_2$\/}. Asian J. Math. \textbf{5} (2001), 183--290.
\bibitem[Zi]{Ziegler} G. Ziegler, {\it Lectures on Polytopes}. Graduate Texts in Mathematics, 152. Springer-Verlag, New York, 1995.
\end{thebibliography}
\end{document}